\documentclass[a4paper,11pt,fleqn]{article}
\usepackage{amsmath}
\usepackage{amssymb}
\usepackage{theorem}
\usepackage{mathrsfs} 
\usepackage{mathabx} 
\usepackage{pstricks}
\usepackage{pstricks-add}
\usepackage{euscript}
\usepackage{graphicx}
\usepackage{enumitem}
\usepackage{charter}
\usepackage{empheq}
\usepackage{caption}
\usepackage{subcaption}
\newcommand{\email}[1]{\href{mailto:#1}{\nolinkurl{#1}}}
\definecolor{dred}{HTML}{D90404}
\definecolor{labelkey}{rgb}{0,0.08,0.45}
\definecolor{refkey}{rgb}{0,0.6,0.0}
\definecolor{Brown}{rgb}{0.45,0.0,0.05}
\definecolor{dgreen}{rgb}{0.00,0.49,0.00}
\definecolor{dblue}{rgb}{0,0.08,0.75}
\RequirePackage[colorlinks,hyperindex]{hyperref} 
\hypersetup{linktocpage=true,citecolor=dblue,linkcolor=dgreen}
\PassOptionsToPackage{normalem}{ulem}
\usepackage{ulem}

\input{lbmenc.def}
\input{lbsenc.def}
\input{lblenc.def}
\DeclareSymbolFont{bsksymbols} {LBM}    {bskms}{m}{n}
\DeclareMathSymbol{\ltimesblack}     {\mathord}{bsksymbols}{"E9}
\renewcommand{\leq}{\ensuremath{\leqslant}}
\renewcommand{\geq}{\ensuremath{\geqslant}}

\renewcommand{\ge}{\ensuremath{\geqslant}}

\newcommand{\pair}[2]{{\left\langle{{#1},{#2}}\right\rangle}}
\newcommand{\menge}[2]{\big\{{#1}~\big |~{#2}\big\}} 

\newcommand{\pbas}[1]{{#1}^{{\tiny\mbox{$\vee$}}}}

\newcommand{\phaut}[1]{{#1}^{{\tiny\mbox{$\wedge$}}}}
\newcommand{\haut}[1]{{#1}^{\small\raisebox{0.0mm}%
{$\blacktriangleup$}}}
\newcommand{\haute}[1]{{#1}^{\small\raisebox{-0.2mm}%
{$\blacktriangleup$}*}}
\newcommand{\ephaut}[1]{{#1}^{*{\tiny\mbox{$\wedge$}}}}

\newcommand{\bas}[1]{{#1}^{\small\raisebox{0.0mm}%
{$\blacktriangledown$}}}
\newcommand{\base}[1]{{#1}^{\small\raisebox{0.0mm}%
{$\blacktriangledown$}*}}
\newcommand{\ebas}[1]{{#1}^{*\small\raisebox{0.0mm}%
{$\blacktriangledown$}}}
\newcommand{\ebase}[1]{{#1}^{*\small\raisebox{0.0mm}%
{$\blacktriangledown$}*}}
\newcommand{\ehaute}[1]{{#1}^{*\small\raisebox{-0.2mm}%
{$\blacktriangleup$}*}}
\newcommand{\ehaut}[1]{{#1}^{*\small\raisebox{-0.2mm}%
{$\blacktriangleup$}}}
\newcommand{\petoile}{p^{\scalebox{0.7}{\ensuremath{*}}}}
\newcommand{\normetoile}{\raisebox{-0.3mm}{\scalebox{0.7}%
{\ensuremath{*}}}}
\newcommand{\rec}{\ensuremath{\text{\rm rec}\,}}
\newcommand{\epi}{\ensuremath{\text{\rm epi}\,}}
\newcommand{\cam}{\ensuremath{\text{\rm cam}\,}}

\newcommand{\UU}{\ensuremath{\mathcal{M}}}
\newcommand{\VV}{\ensuremath{\mathcal{R}}}
\newcommand{\XX}{\ensuremath{\mathcal{X}}}
\newcommand{\YY}{\ensuremath{\mathcal{Y}}}

\newcommand{\emp}{\ensuremath{{\varnothing}}}
\newcommand{\lev}[1]{{\ensuremath{{{{\text{\rm lev}}}_{#1}}\,}}}

\newcommand{\ppersp}{\ensuremath{\ltimes}}
\newcommand{\persp}{\ensuremath{\,\ltimesblack\,}}

\newcommand{\RR}{\ensuremath{\mathbb{R}}}
\newcommand{\RP}{\ensuremath{\left[0,+\infty\right[}}
\newcommand{\RM}{\ensuremath{\left]-\infty,0\right]}}
\newcommand{\RMM}{\ensuremath{\left]-\infty,0\right[}}

\newcommand{\RPP}{\ensuremath{\left]0,+\infty\right[}}
\newcommand{\RPPX}{\ensuremath{\left]0,+\infty\right]}}

\newcommand{\RPX}{\ensuremath{\left[0,+\infty\right]}}
\newcommand{\RX}{\ensuremath{\left]-\infty,+\infty\right]}}
\newcommand{\RXX}{\ensuremath{\left[-\infty,+\infty\right]}}

\newcommand{\NN}{\ensuremath{\mathbb N}}

\newcommand{\exi}{\ensuremath{\exists\,}}

\newcommand{\pinf}{\ensuremath{{+\infty}}}
\newcommand{\minf}{\ensuremath{{-\infty}}}
\newcommand{\dom}{\ensuremath{\text{\rm dom}\,}}
\newcommand{\cdom}{\ensuremath{\overline{\text{\rm dom}}\,}}

\newcommand{\conv}{\ensuremath{\text{\rm conv}\,}}
\newcommand{\cconv}{\ensuremath{\overline{\text{\rm conv}}\,}}

\newcommand{\zeroun}{\ensuremath{\left]0,1\right[}} 
\newcommand{\rzeroun}{\ensuremath{\left]0,1\right]}} 
\newcommand{\lzeroun}{\ensuremath{\left[0,1\right[}} 

\newtheorem{theorem}{Theorem}[section]
\newtheorem{lemma}[theorem]{Lemma}
\newtheorem{corollary}[theorem]{Corollary}
\newtheorem{proposition}[theorem]{Proposition}
\theoremstyle{plain}{\theorembodyfont{\rmfamily}%
}
\theoremstyle{plain}{\theorembodyfont{\rmfamily}%
}
\theoremstyle{plain}{\theorembodyfont{\rmfamily}%
}
\theoremstyle{plain}{\theorembodyfont{\rmfamily}%
}
\theoremstyle{plain}{\theorembodyfont{\rmfamily}%
}
\theoremstyle{plain}{\theorembodyfont{\rmfamily}%
}
\theoremstyle{plain}{\theorembodyfont{\rmfamily}%
\newtheorem{example}[theorem]{Example}}
\theoremstyle{plain}{\theorembodyfont{\rmfamily}%
\newtheorem{remark}[theorem]{Remark}}
\theoremstyle{plain}{\theorembodyfont{\rmfamily}%
\newtheorem{definition}[theorem]{Definition}}
\theoremstyle{plain}{\theorembodyfont{\rmfamily}%
}

\numberwithin{equation}{section}
\def\abstract{\noindent{\bfseries Abstract}. \ignorespaces}

\begin{document}
\title{\sffamily \vskip -9mm 
Perspective Functions with Nonlinear 
Scaling\thanks{Contact author: P. L. Combettes, 
\email{plc@math.ncsu.edu}, phone: +1 919 515 2671. 
The work of P. L. Combettes was supported by the 
National Science Foundation under grant DMS-1818946.}}
\author{Luis M. Brice\~{n}o-Arias$^{1}$, Patrick L. Combettes$^2$, 
and Francisco J. Silva$^3$
\\[4mm]
\small
\small $\!^1$Universidad T\'ecnica Federico Santa Mar\'ia,
Departamento de Matem\'atica, Santiago, Chile\\
\small \email{luis.briceno@usm.cl}\\[3mm]
\small $\!^2$North Carolina State University,
Department of Mathematics, Raleigh, NC 27695-8205, USA\\
\small\email{plc@math.ncsu.edu}\\[4mm]
\small $\!^3$Universit\'e de Limoges, Laboratoire XLIM, 
87060 Limoges, France\\
\small\email{francisco.silva@unilim.fr}
}

\bigskip
\date{~}
\maketitle

\begin{abstract}
The classical perspective of a function is a construction which
transforms a convex function into one that is jointly convex with
respect to an auxiliary scaling variable. Motivated by applications
in several areas of applied analysis, we investigate an extension
of this construct in which the scaling variable is replaced by a
nonlinear term. Our construction is placed in the general context
of locally convex spaces and it generates a lower semicontinuous
convex function under broad assumptions on the underlying
functions. Various convex-analytical properties are established and
closed-form expressions are derived. Several applications are
presented.
\end{abstract} 


\newpage
\section{Introduction}
\label{sec:1}

The objective of this work is to study the following construction,
which combines two functions to generate a lower semicontinuous
convex function on a product space. Throughout, $\XX$ and $\YY$ 
are real, locally convex, Hausdorff topological vector spaces. 

\begin{definition}
\label{d:persp}
The \emph{preperspective} of a \emph{base} function
$\varphi\colon\XX\to\RXX$ with respect to a \emph{scaling} 
function $s\colon\YY\to\RXX$ is
\begin{equation}
\label{e:perspective2-}
\begin{array}{lcll}
\!\!\!\varphi\ppersp s\colon\!\!&\!\!\XX\times\YY&\!\!\to\!\!
&\RXX\\[2mm]
&\!\!(x,y)\!\!&\!\!\mapsto\!\!&
\begin{cases}
s(y)\varphi\bigg(\dfrac{x}{s(y)}\bigg),
&\text{if}\:\:0<s(y)<\pinf;\\
\pinf,&\text{if}\:\:\minf\leq s(y)\leq 0
\;\;\text{or}\:\:s(y)=\pinf,
\end{cases}
\end{array}
\end{equation}
and the \emph{perspective} of $\varphi$ with respect to
$s$ is the largest lower semicontinuous convex function 
$\varphi\persp s$ minorizing $\varphi\ppersp s$.
\end{definition}

Definition~\ref{d:persp} provides a general model for functions
found in areas such as 
mean field games \cite{Ach16a}, 
machine learning \cite{Aval18,Nock16}, 
physics \cite{Berc13,Car22b}, 
optimal transportation \cite{Bren02,Card13,Carr10,Dolb09,Maas11}, 
operator theory \cite{Carl19,Effr09,Niko13}, 
statistics \cite{Ejst20,Owen07}, 
matrix analysis \cite{Daco08}, 
mathematical programming \cite{Imai88,Mare01}, 
information theory \cite{Lutw05,Berc21},
inverse problems \cite{Micc13},
and economics \cite{Zell66}.
Although it appears for instance in \cite{Berc13,Lutw05,Berc21},
the preperspective $\varphi\ppersp s$ is of limited use
in variational problems due to its lack of 
lower semicontinuity and convexity.

Let us note that Definition~\ref{d:persp} covers the classical 
notion of a linearly scaled perspective. Indeed, let
$\Gamma_0(\XX)$ be the class of proper lower semicontinuous convex
functions from $\XX$ to $\RX$. Take $\varphi\in\Gamma_0(\XX)$ and
let $\rec\varphi$ denote the recession function of $\varphi$. Then
the classical perspective of $\varphi$ is
\begin{equation} 
\label{e:0} 
\widetilde{\varphi}
\colon\XX\times\RR\to\RX\colon (x,y)\mapsto 
\begin{cases}
y\varphi\bigg(\dfrac{x}{y}\bigg),&\text{if}\:\:y>0;\\
(\rec\varphi)(x),&\text{if}\:\:y=0;\\ 
\pinf,&\text{if}\:\:y<0.
\end{cases} 
\end{equation}
Upon letting $\YY=\RR$ and $s\colon y\mapsto y$, it follows from
\cite[Theorem~3E]{Rock66} that 
$\varphi\persp s=\widetilde{\varphi}$.
This linear scaling framework is also studied in 
\cite{Svva18,Jmaa18,Hiri93,Rock70}. 

The investigation of
notions of perspectives with nonlinear scaling functions was
initiated in \cite{Mar05a,Mar05b,Mar05c} in Euclidean spaces and
extensions to infinite-dimensional normed spaces were carried out 
in \cite{Zali08}. In these papers, $\varphi\in\Gamma_0(\XX)$ and
either $\varphi(0)\leq 0$ and $-s\in\Gamma_0(\YY)$, or
$\varphi\geq\rec\varphi$ and $s\in\Gamma_0(\YY)$. Such conditions
are not fulfilled for perspectives using the elementary base
function $\varphi=|\cdot|^2+\alpha$ ($\alpha\in\RPP$) on $\XX=\RR$,
which is used in \cite{Anto10} (see \cite{Ejst20,Owen07} for
similar examples). In addition, the construction proposed in
\cite{Mar05a,Mar05b,Mar05c,Zali08} provides lower semicontinuous
convex functions $f\leq\varphi\persp s$ and, when $\YY=\RR$ and
$s\colon y\mapsto y$, it does not capture \eqref{e:0} for a general
$\varphi\in\Gamma_0(\XX)$. 

The goal of the present work is to build a theory of perspective
functions with nonlinear scaling in the context of
Definition~\ref{d:persp} and to derive closed-form expressions for
them. We review notation and preliminary results in
Section~\ref{sec:2}. In Section~\ref{sec:3}, we introduce and study
two notions of functional envelopes which will greatly facilitate
our analysis and will constitute structuring blocks in subsequent
sections. Section~\ref{sec:4} is devoted to the derivation of
properties of preperspective functions and the computation of their
conjugates. Closed-form expressions for perspective functions in
the general setting of Definition~\ref{d:persp} are derived in
Section~\ref{sec:5}, as well as conditions that characterize their
properness. Finally, in Section~\ref{sec:6}, we provide examples
and applications of our results and, in Section~\ref{sec:7}, we
make closing statements.

\section{Notation and preliminary results}
\label{sec:2}

\subsection{Notation}
\label{sec:21}

We recall that, throughout, $\XX$ and $\YY$ 
are real, locally convex, Hausdorff topological vector spaces. 
Let $\XX^*$ be the topological dual of $\XX$, which is equipped
with the weak$*$ topology and is thus also a locally convex
Hausdorff topological vector space. In this context, $\XX$ and
$\XX^*$ are placed in compatible duality (see \cite{Bour81}) via
the canonical form 
$\pair{\cdot}{\cdot}_\XX\colon\XX\times\XX^*\to\RR
\colon(x,x^*)\mapsto x^*(x)$. We denote
by $\XX\oplus\YY$ the standard product vector space equipped with
the product topology and paired with its topological dual 
$\XX^*\times\YY^*$ via 
\begin{equation}
\label{ici}
\big(\forall(x,y)\in\XX\times\YY\big)
\big(\forall(x^*,y^*)\in\XX^*\times\YY^*\big)
\quad\pair{(x,y)}{(x^*,y^*)}_{\XX\times\YY}=
\pair{x}{x^*}_{\XX}+\pair{y}{y^*}_{\YY}.
\end{equation}
From now on, we drop the subscripts on the pairing brackets.
Let $f\colon\XX\to\RXX$. Then
$\dom f=\menge{x\in\XX}{f(x)<\pinf}$ is the domain of $f$,
$\cdom f$ the closure of $\dom f$,
$\lev{\leq\xi}f=\menge{x\in\XX}{f(x)\leq\xi}$ the
lower level set of $f$ at height $\xi\in\RR$, and
$\epi f=\menge{(x,\xi)\in\XX\times\RR}{f(x)\leq\xi}$ the epigraph
of $f$. We say that $f$ is convex if $\epi f$ is convex, lower
semicontinuous if $\epi f$ is closed, and 
proper if $\minf\notin f(\XX)\neq\{\pinf\}$. 
We denote by $\cam f$ the set of continuous affine minorants
of $f$ and put
\begin{equation}
(\forall x\in\XX)\quad f^{**}(x)=\sup_{a\in\cam f}a(x).
\end{equation}
In addition, we denote by $\breve{f}\colon\XX\to\RXX$ the largest
lower semicontinuous convex function majorized by $f$. The
conjugate of $f$ is 
\begin{equation}
\label{e:dconju}
f^*\colon\XX^*\to\RXX\colon
x^*\mapsto\sup_{x\in\XX}\:\big(\pair{x}{x^*}-f(x)\big)
\end{equation}
and the conjugate of $g\colon\XX^*\to\RXX$ is 
\begin{equation}
\label{e:conju+}
g^*\colon\XX\to\RXX\colon
x\mapsto\sup_{x^*\in\XX^*}\:\big(\pair{x}{x^*}-g(x^*)\big).
\end{equation}
If $f$ is proper and convex, its recession function is
\begin{equation}
\rec f\colon\XX\to\RXX\colon x\mapsto
\sup_{y\in\dom f}\big(f(x+y)-f(y)\big).
\end{equation}
and, if $f\in\Gamma_0(\XX)$ and $z\in\dom f$, we have
\begin{equation}
\label{e:uu}
(\forall x\in\XX)\quad(\rec f)(x)=\lim_{0<\alpha\to\pinf}
\dfrac{f(z+\alpha x)-f(z)}{\alpha}=\sup_{\alpha\in\RPP}
\dfrac{f(z+\alpha x)-f(z)}{\alpha}.
\end{equation}
The set of proper lower semicontinuous convex functions from $\XX$
to $\RX$ is denoted by $\Gamma_0(\XX)$.

Let $C$ be a subset of $\XX$. The indicator function of $C$ is
denoted by $\iota_C$, the support function of $C$ by $\sigma_C$,
the smallest convex set containing $C$ by $\conv C$, the
smallest closed convex set containing $C$ by $\cconv C$, and the
recession cone of $C$ by $\rec C$.

\subsection{Facts from convex analysis}
\label{sec:22}

The first three lemmas are standard; see
\cite{Cast77,Ioff74,Laur72,More67,Zali02}.

\begin{lemma}
\label{l:15}
Let $f\in\Gamma_0(\XX)$, let $x\in\dom f$, and let $y\in\dom f$.
Then $f$ is continuous relative to $[x,y]$.
\end{lemma}

\begin{lemma}
\label{l:99}
Let $C\subset\XX$. Then $\sigma_C=\sigma_{\cconv C}$.
\end{lemma}

\begin{lemma}
\label{l:6}
Let $f\colon\XX\to\RXX$. Then the following hold:
\begin{enumerate}
\item
\label{l:6i-} 
$(f^*)^*=f^{**}\leq\breve{f}\leq f$.
\item
\label{l:6i} 
$(\breve{f})^*=f^*=f^{***}$.
\item
\label{l:6ii} 
$\cam f\neq\emp$ $\Leftrightarrow$ 
$\minf\notin f^{**}(\XX)$ $\Leftrightarrow$ $f^{**}\not\equiv\minf$
$\Leftrightarrow$ $\dom f^*\neq\varnothing$.
\item
\label{l:6iii} 
$\cam f=\emp$ $\Rightarrow$ $\breve{f}(\XX)\subset\{\minf,\pinf\}$.
\item
\label{l:6iv} 
Suppose that $f\in\Gamma_0(\XX)$. Then $\cam f\neq\emp$ and
$\breve{f}=f^{**}=f$.
\end{enumerate}
\end{lemma}

\begin{lemma}
\label{l:9}
Let $f\colon\XX\to\RX$ be such that $\cam f\neq\emp$.
Then the following hold:
\begin{enumerate}
\item
\label{l:9i-} 
Suppose that $f\not\equiv\pinf$. Then 
$f^*\in\Gamma_0(\XX^*)$ and $f^{**}\in\Gamma_0(\XX)$. 
\item
\label{l:9i} 
$\breve{f}=f^{**}$.
\item
\label{l:9ii}
$\cconv\dom f=\cdom f^{**}$.
\end{enumerate}
\end{lemma}
\begin{proof}
\ref{l:9i-}--\ref{l:9i}: See \cite{More67}.

\ref{l:9ii}: We derive from \ref{l:9i} and
\cite[Proposition~9.8(iv)]{Livre1} (its proof
remains valid in our setting) that 
$\conv\dom f\subset\dom f^{**}\subset\cconv\dom f$. Taking the
closure yields the identity. 
\end{proof}

\begin{lemma}
\label{l:pjlaurent}
Let $f\in\Gamma_0(\XX)$. Then the following hold:
\begin{enumerate}
\item
\label{l:pjlaurenti}
$\rec\epi f=\epi\rec f$.
\item
\label{l:pjlaurentii}
$\rec f=\sigma_{\dom f^*}=\rec(f^{**})$.
\item
\label{l:pjlaurentiv} 
$f=\rec{f}$ $\Leftrightarrow$ $f^*(\dom f^*)=\{0\}$.
\end{enumerate}
\end{lemma}
\begin{proof}
\ref{l:pjlaurenti}: 
See \cite[Proposition~9.29]{Livre1} 
(its proof remains valid in our setting).

\ref{l:pjlaurentii}: The first identity is from 
\cite[Corollary~3D]{Rock66}. In view of Lemma~\ref{l:6}\ref{l:6iv},
it implies the second.

\ref{l:pjlaurentiv}: 
It follows from Lemma~\ref{l:6}\ref{l:6iv},
\ref{l:pjlaurentii}, and Lemma~\ref{l:9}\ref{l:9i} that 
\begin{equation}
f=\rec f\quad\Leftrightarrow\quad 
f^*=(\iota_{\dom f^*})^{**}=\iota_{\cdom f^*},
\end{equation}
which implies that $\dom f^*=\dom\iota_{\cdom f^*}=\cdom f^*$.
Thus, since $f^*$ is lower semicontinuous, 
$f=\rec f$ $\Leftrightarrow$ 
$f^*=\iota_{\cdom f^*}=\iota_{\dom f^*}$
$\Leftrightarrow$ $f^*(\dom f^*)=\{0\}$.
\end{proof}

\begin{lemma}
\label{l:12e} 
Let $f\in\Gamma_0(\XX)$. Then the following are equivalent:
\begin{enumerate}
\item 
\label{l:12ei} 
$\rec f\leq f$.
\item 
\label{l:12eii} 
$(\forall\lambda\in\left[1,\pinf\right[)(\forall x\in\XX)$
$f(\lambda x)\leq\lambda f(x)$.
\item 
\label{l:12eiii} 
$f^*(\dom f^*)\subset\RM$.
\end{enumerate}
\end{lemma}
\begin{proof}
\ref{l:12ei}$\Rightarrow$\ref{l:12eii}:
Without loss of generality, let $\lambda\in\left]1,\pinf\right[$
and $x\in\dom f$.
Arguing as in the proof of \cite[Proposition~8(iii)]{Zali08},
we observe that Lemma~\ref{l:pjlaurent}\ref{l:pjlaurenti} yields
\begin{equation}
(\lambda-1)\epi f\subset(\lambda-1)\epi\rec f
=(\lambda-1)\rec\epi f=\rec\epi f.
\end{equation}
Hence, $(\lambda x,\lambda f(x))
=(x,f(x))+(\lambda-1)(x,f(x))\in\epi f+\rec\epi f=\epi f$.
Therefore, $f(\lambda x)\leq\lambda f(x)$.

\ref{l:12eii}$\Rightarrow$\ref{l:12eiii}: 
Let $x^*\in\dom f^*$. Since
\begin{align}
f^*(x^*)
&=\sup_{x\in\XX}\big(\pair{x}{x^*}-f(x)\big)\nonumber\\
&=\sup_{y\in\XX}\sup_{\lambda\in\left[1,\pinf\right[}
\big(\pair{\lambda y}{x^*}-f(\lambda y)\big)\nonumber\\
&\geq\sup_{\lambda\in\left[1,\pinf\right[}\;\lambda\,\sup_{y\in\XX}
\big(\pair{y}{x^*}-f(y)\big)\nonumber\\
&=\sup_{\lambda\in\left[1,\pinf\right[}\lambda\:
f^*(x^*),
\end{align}
we have $f^*(x^*)>0$ $\Rightarrow$ $f^*(x^*)=\pinf$, which
contradicts the fact that $x^*\in\dom f^*$.

\ref{l:12eiii}$\Rightarrow $\ref{l:12ei}:
In view of Lemma~\ref{l:pjlaurent}\ref{l:pjlaurentii}, 
$(\forall x^*\in\dom f^*)$ $f^*(x^*)\in\RM$
$\Rightarrow$
$(\forall x^*\in\dom f^*)$ $x^*\leq f$ 
$\Rightarrow$ $\sigma_{\dom f ^*}\leq f$
$\Rightarrow$ $\rec f\leq f$.
\end{proof}

\section{The $\blacktriangledown$ and $\blacktriangleup$ 
envelopes}
\label{sec:3}

We introduce two types of envelope of a function that
will be essential in our analysis.

\begin{definition}
\label{d:hautbas}
Let $f\colon\XX\to\RXX$. Then 
\begin{equation}
\label{e:pbas}
\pbas{f}\colon\XX\to\RX\colon x\mapsto
\begin{cases}
f(x),&\text{if}\:\:\minf<f(x)<0;\\
\pinf,&\text{otherwise}
\end{cases}
\end{equation}
and the $\blacktriangledown$ envelope of $f$ is
\begin{equation}
\label{e:bas}
\bas{f}=f^{{\tiny\mbox{$\vee$}}**}.
\end{equation}
Furthermore, 
\begin{equation}
\label{e:phaut}
\phaut{f}\colon\XX\to\RX\colon x\mapsto
\begin{cases}
f(x),&\text{if}\:\:0<f(x)<\pinf;\\
\pinf,&\text{otherwise}
\end{cases}
\end{equation}
and the $\blacktriangleup$ envelope of $f$ is
\begin{equation}
\label{e:haut}
\haut{f}=f^{{\tiny\mbox{$\wedge$}}**}.
\end{equation}
\end{definition}

Let us examine some key properties of these envelopes.

\begin{lemma}
\label{l:12}
Let $f\colon\XX\mapsto\RXX$ be such that $E=f^{-1}(\RMM)\neq\emp$.
Then the following holds:
\begin{enumerate}
\item
\label{l:120+} 
Suppose that $\cam\pbas{f}\neq\emp$.
Then $\bas{f}\in\Gamma_{0}(\XX)$,
$\dom\bas{f}\subset\cconv E$, and 
$\bas{f}(\dom\bas{f})\subset\RM$. 
\end{enumerate}
Now suppose that, in addition, $f\in\Gamma_0(\XX)$. Then the
following are satisfied:
\begin{enumerate}[resume]
\item
\label{l:12i} 
$\overline{E}=f^{-1}(\RM)$.
\item
\label{l:12ii} 
$\bas{f}=f+\iota_{f^{-1}(\RM)}$.
\item
\label{l:12ii+} 
$\dom\bas{f}=f^{-1}(\RM)$.
\item
\label{l:12v} 
${\bas{f}}^{-1}(\{0\})=f^{-1}(\{0\})$.
\item
\label{l:12iii} 
$E=(\bas{f})^{-1}(\RMM)$ and $\bas{f}|_E=f|_E$.
\end{enumerate}
\end{lemma}
\begin{proof} 
\ref{l:120+}: 
The fact that $\bas{f}\in\Gamma_{0}(\XX)$ follows from 
\eqref{e:bas} and Lemma~\ref{l:9}\ref{l:9i-}.
Next, since ${\pbas{f}}(\XX)\subset\RMM\cup\{\pinf\}$, 
we deduce from \eqref{e:bas}, Lemma~\ref{l:9}\ref{l:9ii}, and 
\eqref{e:pbas} that
\begin{align}
\dom\bas{f}&\subset\cdom\bas{f}\nonumber\\
&=\cconv\,\dom\pbas{f}\nonumber\\
&=\cconv(\pbas{f})^{-1}(\RMM)\label{e:1212dd}\\
&=\cconv E.
\end{align}
On the other hand, Lemma~\ref{l:6}\ref{l:6i-} 
yields $\bas{f}\leq\pbas{f}$. Hence, we derive 
from \eqref{e:1212dd} that
\begin{align}
\dom\bas{f}& \subset
\cconv(\bas{f})^{-1}(\RMM)\nonumber\\
&\subset\cconv(\bas{f})^{-1}(\RM)\nonumber\\
&=(\bas{f})^{-1}(\RM)
\end{align}
since $\bas{f}\in\Gamma_{0}(\XX)$.

\ref{l:12i}: Since $f$ is lower semicontinuous, 
$f^{-1}(\RM)$ is closed. Therefore
$E\subset f^{-1}(\RM)$ $\Rightarrow$
$\overline{E}\subset f^{-1}(\RM)$. Conversely, take 
$x_0\in f^{-1}(\RM)$ and $x\in E$, and set
$(\forall\alpha\in\zeroun)$ $x_\alpha=\alpha x+(1-\alpha)x_0$.
Since $f$ is convex $(\forall\alpha\in\zeroun)$ 
$f(x_\alpha)\leq\alpha f(x)+(1-\alpha)f(x_0)<0$, hence
$x_\alpha\in E$. Thus 
$x_0=\lim_{\alpha\downarrow 0}x_\alpha\in\overline{E}$.

\ref{l:12ii}: 
Let $x^*\in\XX^*$ and let us show that
$f^{{\tiny\mbox{$\vee$}}*}(x^*)=\sup(x^*-f)(\overline{E})$. Since
$f^{{\tiny\mbox{$\vee$}}*}(x^*)=\sup(x^*-f)(E)$, we have
$f^{{\tiny\mbox{$\vee$}}*}(x^*)\leq\sup(x^*-f)(\overline{E})$. 
To get the reverse inequality let 
$x\in\overline{E}$. We need to show that
$\pair{x}{x^*}-f(x)\leq f^{{\tiny\mbox{$\vee$}}*}(x^*)$. 
It is enough to assume that
$x\in\overline{E}\smallsetminus E$, which yields $f(x)=0$.
In addition, since $x^*$ is lower semicontinuous and 
$\pbas{f}|_E<0$, 
\begin{equation}
\pair{x}{x^*}-f(x)=\pair{x}{x^*}\leq\sup x^*(\overline{E})
=\sup x^*(E)\leq f^{{\tiny\mbox{$\vee$}}*}(x^*).
\end{equation}
Thus,
\begin{equation}
f^{{\tiny\mbox{$\vee$}}*}(x^*)
=\sup(x^*-f)(\overline{E})
=\big(f+\iota_{\overline{E}}\big)^*(x^*)
=\big(f+\iota_{f^{-1}(\RM)}\big)^*(x^*).
\end{equation}
On the other hand, since $E\neq\emp$, using \ref{l:12i}, we see
that
\begin{equation}
\label{e:vdg}
f+\iota_{f^{-1}(\RM)}=f+\iota_{\overline{E}}\in\Gamma_{0}(\XX).
\end{equation}
Altogether, \eqref{e:vdg} and Lemma~\ref{l:6}\ref{l:6iv} yield
$\bas{f}=f^{{\tiny\mbox{$\vee$}}**}=(f+\iota_{f^{-1}(\RM)})^{**}=
f+\iota_{f^{-1}(\RM)}$.

\ref{l:12ii+}--\ref{l:12iii}: These follow from \ref{l:12ii}.
\end{proof}

\begin{lemma}
\label{l:24}
Let $f\colon\XX\to\RXX$ be such that $F=f^{-1}(\RPP)\neq\emp$. 
Then the following holds:
\begin{enumerate}
\item 
\label{l:24ii+}
$\haut{f}\in\Gamma_0(\XX)$, $\dom\haut{f}\subset\cconv F$, and
$\haut{f}(\dom\haut{f})\subset\RP$.
\end{enumerate}
Now suppose that, in addition, $f\in\Gamma_0(\XX)$. Then the
following are satisfied:
\begin{enumerate}[resume]
\item
\label{l:24ii} 
$\haut{f}=\max\{f,0\}+\iota_{\cconv F}$.
\item
\label{l:24iv} 
$\dom\haut{f}=\dom f\cap\cconv F\supset F$. 
\item
\label{l:24ii++} 
$(\haut{f})^{-1}(\{0\})=f^{-1}(\RM)\cap\cconv F$.
\item
\label{l:24v}
$(\haut{f})^{-1}(\{0\})=\emp$ $\Leftrightarrow$ $f^{-1}(\RM)=\emp$.
\item
\label{l:24iii} 
$F=(\haut{f})^{-1}(\RPP)$ and $f|_{F}=\haut{f}|_{F}$.
\end{enumerate}
\end{lemma}
\begin{proof}
\ref{l:24ii+}: 
Set $\theta\colon\XX\to\RR\colon x\mapsto 0$.
Since $\phaut{f}>\theta\in\cam f$ and $F\neq\emp$,
Lemma~\ref{l:9}\ref{l:9i-} asserts that $\haut{f}\in\Gamma_0(\XX)$.
In addition, \eqref{e:phaut}, \eqref{e:haut}, and
Lemma~\ref{l:9}\ref{l:9ii} yield
\begin{equation}
\label{e:98++}
\dom\haut{f}\subset
\cdom\haut{f}=\cconv\dom\phaut{f}=\cconv F 
\end{equation}
and
\begin{equation}
\label{e:81++}
\big(\forall x\in\dom\haut{f}\big)\quad 
0=\theta(x)=\theta^{**}(x)\leq f^{{\tiny\mbox{$\wedge$}}**}(x)
=\haut{f}(x)<\pinf.
\end{equation}

\ref{l:24ii}: 
Set $g=\max\{f,0\}$. 
Since $F\neq\emp$, we have $g\in\Gamma_0(\XX)$ and
$\pinf\not\equiv\phaut{f}=g+\iota_F\geq g+\iota_{\cconv F}
\in\Gamma_0(\XX)$. Hence, appealing to Lemma~\ref{l:6}\ref{l:6iv},
we obtain
\begin{equation}
g+\iota_{\cconv{F}}\leq\haut{f}\leq g+\iota_{F}. 
\end{equation} 
Let $x\in\XX$. If $x\in F$, then
\begin{equation}
\label{e:g46}
g(x)+\iota_{\cconv{F}}(x)=\haut{f}(x)=
g(x)+\iota_{F}(x)=g(x). 
\end{equation}
If $x\notin\cconv F$ or $x\notin\dom g$, then
$g(x)+\iota_{\cconv{F}}(x)=\haut{f}(x)=
g(x)+\iota_{F}(x)=\pinf$. Now, suppose that 
$x\in(\dom g\cap\cconv F)\smallsetminus F$. Then, since
$g(\XX\smallsetminus F)\subset\{0,\pinf\}$, we have 
\begin{equation}
\label{e:gzero}
g(x)=0.
\end{equation} 
It remains to show that $\haut{f}(x)=0$. To this end, fix 
$\varepsilon\in\RPP$. Suppose first that
$x\in(\conv F)\smallsetminus F$. 
Since $x\in\conv F$, there exist finite families 
$(x_i)_{i\in I}$ in $F$ and $(\alpha_i)_{i\in I}$ in $\zeroun$
such that $\sum_{i\in I}\alpha_i=1$ and 
$x=\sum_{i\in I}\alpha_i x_i$. Hence,
it follows from Lemma~\ref{l:15}, \eqref{e:g46},
and \eqref{e:gzero} that, for every 
$i\in I$, there exists $z_i\in\left]x,x_i\right[\cap F$ 
such that $\haut{f}(z_i)=g(z_i)\in\left]0,\varepsilon\right]$,
say $z_i=(1-\eta_i)x+\eta_ix_i$ for some $\eta_i\in\zeroun$.
Therefore, for every $i\in I$,
$x_i=\eta_i^{-1}z_i+(1-\eta_i^{-1})x$.
In turn, $x=\sum_{i\in I}\alpha_i x_i=\sum_{i\in I}\beta_i z_i$,
where, for every $i\in I$, 
$\beta_i=\alpha_i\eta_i^{-1}/(\sum_{j\in I}\alpha_j\eta_j^{-1})>0$.
Since $\sum_{i\in I}\beta_i=1$ and 
$\{z_i\}_{i\in I}\subset\lev{\leq\varepsilon}g$, we have
$x\in\lev{\leq\varepsilon}g$ and
$0\leq\haut{f}(x)\leq\sum_{i\in I}\beta_i\haut{f}(z_i)
=\sum_{i\in I}\beta_i g(z_i)\leq\varepsilon$. 
Thus, $\haut{f}(x)=0$. 
Altogether, in view of \eqref{e:gzero},
since $x$ is arbitrarily chosen in 
$(\conv F)\smallsetminus F$, we have 
\begin{equation}
\label{e:passing}
\big(\forall u\in(\conv F)\smallsetminus F\big)\quad g(u)=0
\quad\text{and}\quad 
\haut{f}(u)=0.
\end{equation}
Next, suppose that $x\in(\cconv F)\smallsetminus\conv F$.
Then there exists a net
$(u_a)_{a\in A}$ in $\conv F$ such that $u_a\to x$. For 
every $a\in A$, we consider the following alternatives.
\begin{itemize}
\item 
$u_a\in F$: Since $g(x)=0$ and $g(u_a)\in\RPP$, \eqref{e:g46}
and Lemma~\ref{l:15} guarantee the existence of 
$\widetilde{u}_a\in\left]x,u_a\right[\cap F$ such that 
$\haut{f}(\widetilde{u}_a)=
g(\widetilde{u}_a)\in\left]0,\varepsilon\right]$.
\item 
$u_a\notin F$: Set $\widetilde{u}_a=u_a$. It follows 
from \eqref{e:passing} that $g(\widetilde{u}_a)=0$
and $\haut{f}(\widetilde{u}_a)=0$. 
\end{itemize}
By construction, for every $a\in A$, 
$\widetilde{u}_a\in\conv F$ and, 
if $u_a$ is in a convex neighborhood of $x$, so is
$\widetilde{u}_a$. Since $\XX$ is locally convex, we obtain
$\widetilde{u}_a\to x$. By lower semicontinuity of $\haut{f}$, we
conclude that $0\leq\haut{f}(x)\leq\varliminf
\haut{f}(\widetilde{u}_a)\leq\varepsilon$. This shows that 
$\haut{f}(x)=0$.

\ref{l:24iv}\&\ref{l:24ii++}: These follow from \ref{l:24ii}.

\ref{l:24v}: 
Suppose that $f^{-1}(\RM)\neq\emp$, let $x\in f^{-1}(\RM)$, 
and let $z\in F$. By Lemma~\ref{l:15}, 
$[x,z[\,\cap\,f^{-1}(\{0\})\cap\overline{F}\neq\emp$
and, hence, \ref{l:24ii++} yields $(\haut{f})^{-1}(\{0\})\neq\emp$
since $\overline{F}\subset\cconv F$.
The reverse implication is clear by \ref{l:24ii++}.

\ref{l:24iii}: 
These follow from \ref{l:24ii}.
\end{proof}

\begin{remark} 
In the setting of Lemma~\ref{l:24}, we can have 
$f\in\Gamma_0(\XX)$ and $(\cconv F)\cap f^{-1}(\RMM)\neq\emp$.
Take, for instance, $\XX=\RR^2$, and set 
\begin{equation}
f\colon\XX\to\RX\colon(\xi,\eta)\mapsto
\begin{cases}
\xi^2/\eta-1,&\text{if}\:\:\eta>0;\\
-1,&\text{if}\:\:\xi=\eta=0;\\
\pinf,&\text{otherwise}.
\end{cases}
\end{equation}
Since $f+1$ is an instance of \eqref{e:0}, we have
$f\in\Gamma_0(\XX)$.
For every $n\in\NN$, setting $x_n=(2^{-n},2^{-2n-1})$ yields
$f(x_n)=1$. We obtain $F\ni x_n\to(0,0)\in\cconv F$ and 
$f(0,0)=-1$.
\end{remark}

\begin{lemma} 
\label{l:14}
Let $f\colon\XX\to\RXX$ be such that $F=f^{-1}(\RPP)\neq\emp$ and
assume that $\cam\pbas{(-f)}\neq\emp$. Then the following hold: 
\begin{enumerate}
\item
\label{l:14i-}
$-\bas{(-f)}<\pinf$. 
\item
\label{l:14i}
$0\leq\haut{f}|_{\cconv F}\leq -\bas{(-f)}|_{\cconv F}$. 
\item
\label{l:14ii}
$\dom\haut{f}=\cconv F$.
\end{enumerate}
\end{lemma}
\begin{proof} 
\ref{l:14i-}: 
Since $\cam\pbas{(-f)}\neq\emp$, \eqref{e:bas}
and Lemma~\ref{l:6}\ref{l:6ii} yield $\minf\notin\bas{(-f)}(\XX)$
and therefore $-\bas{(-f)}<\pinf$. 

\ref{l:14i}: 
The first inequality follows from Lemma~\ref{l:24}\ref{l:24ii+}.
We derive from Definition~\ref{d:hautbas} and
Lemma~\ref{l:6}\ref{l:6i-} that
\begin{equation}
\label{e:b4}
(\forall x\in F)\quad\haut{f}(x)\leq\phaut{f}(x)
=-\big(-f(x)\big)=-\pbas{(-f)}(x)\leq-\bas{(-f)}(x). 
\end{equation}
Now set $h=\haut{f}+\bas{(-f)}$. Then \eqref{e:b4} implies that 
$h|_{F}\leq 0$. Since $F\subset\lev{\leq 0}h$ and $h$ is lower
semicontinuous and convex, note that
$\cconv F\subset\cconv\lev{\leq 0}h=\lev{\leq 0}h$.

\ref{l:14ii}: This follows from \ref{l:14i-}, \ref{l:14i}, and
Lemma~\ref{l:24}\ref{l:24ii+}.
\end{proof}

\begin{remark}
\label{r:0}
Let $f\in\Gamma_0(\XX)$ be such that 
$(f^*)^{-1}(\RMM)\neq\emp$. Then
$f=\max\{\ebase{f},\ehaute{f}\}$. Indeed, since 
Lemma~\ref{l:6}\ref{l:6iv} asserts that $f=f^{**}$, it
follows from Lemma~\ref{l:9}\ref{l:9i-},
Lemma~\ref{l:12}\ref{l:12ii}, \eqref{e:phaut},
Lemma~\ref{l:6}\ref{l:6i}, and \eqref{e:haut} that
\begin{align}
\label{e:r0}
(\forall x\in\XX)\quad f(x)
&=\sup_{x^*\in\XX^*}\big(\pair{x}{x^*}-f^*(x^*)\big)\nonumber\\
&=\max\bigg\{
\sup_{(f^*)^{-1}(\RM)}\big(\pair{x}{x^*}-f^*(x^*)\big),
\sup_{(f^*)^{-1}(\RPP)}\big(\pair{x}{x^*}-f^*(x^*)\big)\bigg\}
\nonumber\\
&=\max\bigg\{
\sup_{x^*\in\XX^*}\big(\pair{x}{x^*}-\ebas{f}(x^*)\big),
\sup_{x^*\in\XX^*}\big(\pair{x}{x^*}-\ephaut{f}(x^*)\big)\bigg\}
\nonumber\\
&=\max\Big\{\ebase{f}(x),\ehaute{f}(x)\Big\}.
\end{align}
\end{remark}

\begin{example}
\label{ex:normp}
Suppose that $(\XX,\|\cdot\|)$ is a nonzero real reflexive Banach
space with dual norm 
$\|\cdot\|_{\normetoile}$, let $\alpha\in\RPP$, let 
$p\in\left]1,\pinf\right[$, set $\petoile=p/(p-1)$, and set 
$f=\|\cdot\|^p/p+\alpha^{\petoile}/\petoile$.
Then $f^*=(\|\cdot\|_{\normetoile}^{\petoile}-
\alpha^{\petoile})/\petoile$, 
which yields
$(f^*)^{-1}(\RMM)\neq\emp$ and $(f^*)^{-1}(\RPP)\neq\emp$.
Therefore, since $\cconv (f^{*})^{-1}(\RPP)=\XX^*$,
Lemma~\ref{l:12}\ref{l:12ii} and Lemma~\ref{l:24}\ref{l:24ii} imply
that 
\begin{equation}
\label{e:r9}
\ebas{f}\colon x^*\mapsto\begin{cases}\pinf, 
&\text{if}\:\:\|x^*\|_{\normetoile}>\alpha;\\[2mm]
\dfrac{\|x^*\|_{\normetoile}^{\petoile}
-\alpha^{\petoile}}{\petoile},
&\text{if}\:\:\|x^*\|_{\normetoile}\leq\alpha
\end{cases}
\quad\text{and}\quad\ehaut{f}\colon x^*\mapsto 
\begin{cases}
\dfrac{\|x^*\|_{\normetoile}^{\petoile}-
\alpha^{\petoile}}{\petoile}, 
&\text{if}\:\:\|x^*\|_{\normetoile}>\alpha;\\[2mm]
0,&\text{if}\:\:\|x^*\|_{\normetoile}\leq\alpha.
\end{cases}
\end{equation}
It is noteworthy that we obtain by conjugation 
\begin{equation}
\label{e:r8}
\ebase{f}\colon x\mapsto 
\begin{cases}
\alpha\|x\|,&\;\;\text{if}\:\:\|x\|>\alpha^{\frac{1}{p-1}};\\[2mm]
\dfrac{\|x\|^p}{p}+\dfrac{\alpha^{\petoile}}
{\petoile},&\;\;
\text{if}\:\:\|x\|\leq\alpha^{\frac{1}{p-1}}
\end{cases}\quad\text{and}\quad\ehaute{f}\colon x\mapsto 
\begin{cases}
\dfrac{\|x\|^p}{p}+\dfrac{\alpha^{\petoile}}{\petoile},&
\text{if}\;\;\|x\|>\alpha^{\frac{1}{p-1}};\\[5pt]
\alpha\|x\|,&\text{if}\;\;\|x\|\leq\alpha^{\frac{1}{p-1}}.
\end{cases}
\end{equation}
We recognize, respectively, the $p$th order Huber and Berhu
functions used in \cite{Ejst20,Lium20} (see Figure~\ref{fig:1}).
\end{example}

\begin{figure}
\begin{center}
\includegraphics[scale=0.25]{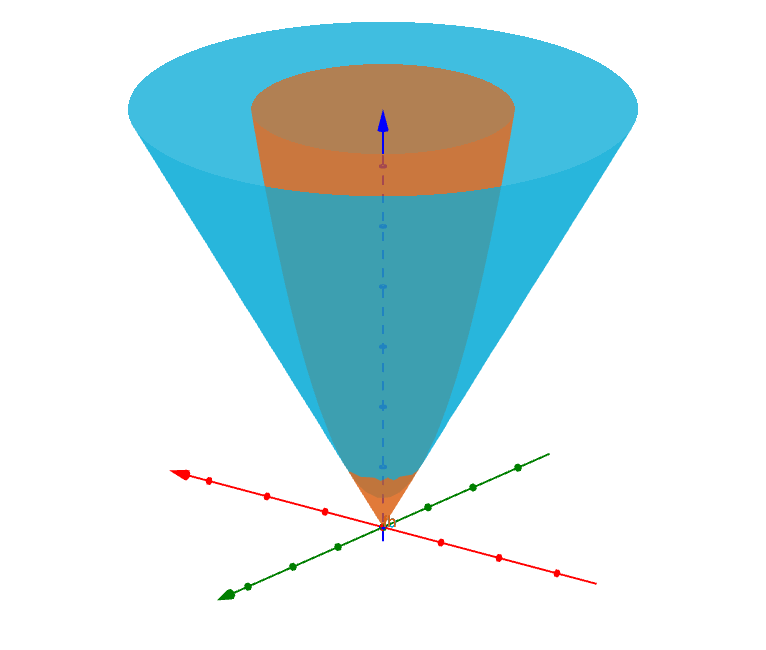}
\end{center}
\caption{Plots of $\ebase{f}$ (Huber, blue) and $\ehaute{f}$
(Berhu, orange) when $\XX=\RR^2$ and $f=(\|\cdot\|_2^2+1)/2$.
We verify that $f$ is the maximum of both functions, as observed 
in Remark~\ref{r:0}.}
\label{fig:1}
\end{figure}

\section{Preperspective functions} 
\label{sec:4}

Let us first record some direct consequences of 
Definition~\ref{d:persp}.

\begin{proposition}
\label{p:28}
Let $\varphi\colon\XX\to\RXX$, let $s\colon\YY\to\RXX$, and set
$S=s^{-1}(\RPP)$. Then the following hold:
\begin{enumerate}
\item 
\label{p:28i} 
$\dom(\varphi\ppersp s)=
\menge{(x,y)\in\XX\times S}{x\in s(y)\dom\varphi}$.
\item
\label{p:28ii} 
$\varphi\ppersp s$ is proper if and only if $\varphi$ is proper
and $S\neq\emp$.
\end{enumerate}
\end{proposition}
\begin{proof}
\ref{p:28i}:
Clear from Definition~\ref{d:persp}.

\ref{p:28ii}: 
We derive from \eqref{e:perspective2-} that 
$\minf\in(\varphi\ppersp s)(\XX\times\YY)$ 
$\Leftrightarrow$ $\minf\in\varphi(\XX)$.
Suppose that $\varphi\ppersp s$ is proper and let 
$(x,y)\in\dom(\varphi\ppersp s)$. 
In view of \ref{p:28i}, $y\in S$ and
$x/s(y)\in\dom\varphi$. Now suppose that
$\varphi$ is proper and $S\neq\emp$, and let
$(x,y)\in\dom\varphi\times S$.
Then $(s(y)x,y)\in\dom(\varphi\ppersp s)$.
\end{proof}

Our first result provides conditions under which the preperspective
of a convex function is itself convex.

\begin{proposition}
\label{p:30}
Let $\varphi\colon\XX\to\RXX$ be convex, let
$s\colon\YY\to\RXX$, set $S=s^{-1}(\RPP)$, and suppose
that one of the following holds:
\begin{enumerate}
\item
\label{p:30iii}
$\varphi$ satisfies
\begin{equation}
\label{e:passeport}
(\forall\lambda\in\left]1,\pinf\right[)(\forall x\in\dom\varphi)
\quad\varphi(\lambda x)\leq\lambda\varphi(x),
\end{equation}
$s$ is proper and convex, and $S$ is convex.
\item
\label{p:30ii}
$\varphi(0)\leq 0$ and $-s$ is proper and convex.
\item
\label{p:30i}
$s$ is an affine function.
\end{enumerate}
Then $\varphi\ppersp s$ is convex.
\end{proposition}
\begin{proof}
Let $\alpha\in\zeroun$, and suppose that 
$(x_1,y_1)\in\dom(\varphi\ppersp s)$ and
$(x_2,y_2)\in\dom(\varphi\ppersp s)$. Set
\begin{equation}
y=\alpha y_1+(1-\alpha)y_2.
\end{equation}
Observe that, since $S$ is convex, $y\in S$. Further, set
\begin{equation}
\label{e:after8}
\beta_1=\dfrac{\alpha s(y_1)}{s(y)},\;\;
\beta_2=\dfrac{(1-\alpha) s(y_2)}{s(y)},\quad\text{and}\quad
\beta=\beta_1+\beta_2,
\end{equation}
and note that $\beta_1\in\RPP$ and $\beta_2\in\RPP$. 

\ref{p:30iii}: Observe that the convexity of $s$ yields
$\beta\in\left[1,\pinf\right[$. In view of
\eqref{e:after8}, \eqref{e:passeport}, and the convexity of
$\varphi$, we have
\begin{align}
(\varphi\ppersp s)\big(\alpha(x_1,y_1)+
(1-\alpha)(x_2,y_2)\big)
&=s(y)\varphi\bigg(\dfrac{\alpha x_1+(1-\alpha)x_2}
{s(y)}\bigg)\nonumber\\
&=s(y)\varphi\bigg(
\dfrac{\beta_1x_1}{s(y_1)}+\dfrac{\beta_2x_2}{s(y_2)}\bigg)
\nonumber\\
&=s(y)\varphi\bigg(\beta\bigg(
\dfrac{\beta_1x_1}{\beta s(y_1)}+\dfrac{\beta_2x_2}{\beta s(y_2)}
\bigg)\bigg)
\nonumber\\
&\leq s(y)\beta\varphi\bigg(
\dfrac{\beta_1}{\beta}\dfrac{x_1}{s(y_1)}+\dfrac{\beta_2}{\beta}
\dfrac{x_2}{s(y_2)}\bigg)\nonumber\\
&\leq s(y)\beta_1\varphi\bigg(
\dfrac{x_1}{s(y_1)}\bigg)+s(y)\beta_2\varphi\bigg(
\dfrac{x_2}{s(y_2)}\bigg)\nonumber\\
&=\alpha s(y_1)\varphi\bigg(
\dfrac{x_1}{s(y_1)}\bigg)+(1-\alpha)s(y_2)\varphi\bigg(
\dfrac{x_2}{s(y_2)}\bigg)\nonumber\\
&=\alpha(\varphi\ppersp s)(x_1,y_1)+(1-\alpha)
(\varphi\ppersp s)(x_2,y_2).
\end{align}

\ref{p:30ii}--\ref{p:30i}:
By convexity, $s(y)\geq\alpha s(y_1)+(1-\alpha)s(y_2)>0$
and, therefore, \eqref{e:after8} yields $\beta\in\rzeroun$. We have
\begin{align}
\label{e:soyunid}
(\varphi\ppersp s)\big(\alpha(x_1,y_1)+
(1-\alpha)(x_2,y_2)\big)&=s(y)\varphi\bigg(\dfrac{\alpha 
x_1+(1-\alpha)x_2}
{s(y)}\bigg)\nonumber\\
&=s(y)\varphi\bigg(
\beta_1\dfrac{x_1}{s(y_1)}+\beta_2\dfrac{x_2}{s(y_2)}
+(1-\beta)0\bigg).
\end{align}
In case \ref{p:30i} we have $\beta=1$ and hence, by convexity of
$\varphi$,
\begin{align}
\label{e:lyon}
(\varphi\ppersp s)\big(\alpha(x_1,y_1)+
(1-\alpha)(x_2,y_2)\big)
&\leq\alpha s(y_1)\varphi\bigg(\dfrac{x_1}{s(y_1)}\bigg)
+(1-\alpha)s(y_2)\varphi\bigg(\dfrac{x_2}{s(y_2)}\bigg)
\nonumber\\
&=\alpha(\varphi\ppersp s)(x_1,y_1)+(1-\alpha)
(\varphi\ppersp s)(x_2,y_2).
\end{align}
We now turn to \ref{p:30ii}. If $\beta=1$, then we obtain 
\eqref{e:lyon} using 
\eqref{e:soyunid}.
On the other hand, if $\beta\in\zeroun$, then since
$\varphi(0)\leq 0$, we have
$(1-\beta)s(y)\varphi(0)\leq 0$. Hence, it follows from
\eqref{e:soyunid} and convexity of 
$\varphi$ that
\begin{align}
(\varphi\ppersp s)\big(\alpha(x_1,y_1)+
(1-\alpha)(x_2,y_2)\big)
&\leq\alpha
s(y_1)\varphi\bigg(\dfrac{x_1}{s(y_1)}\bigg)
+(1-\alpha)s(y_2)\varphi\bigg(\dfrac{x_2}{s(y_2)}\bigg)\nonumber\\
&\quad\;+(1-\beta)s(y)\varphi(0)\nonumber\\
&\leq\alpha s(y_1)\varphi\bigg(\dfrac{x_1}{s(y_1)}\bigg)
+(1-\alpha)s(y_2)\varphi\bigg(\dfrac{x_2}{s(y_2)}\bigg)
\nonumber\\
&\leq\alpha(\varphi\ppersp s)(x_1,y_1)+(1-\alpha)
(\varphi\ppersp s)(x_2,y_2),
\end{align}
which concludes the proof.
\end{proof} 

Next, we determine the conjugate of the preperspective, using the
$\blacktriangledown$ and $\blacktriangleup$ envelopes of
Definition~\ref{d:hautbas}. In view of \eqref{e:perspective2-}, if
$s^{-1}(\RPP)=\emp$, then $(\varphi\ppersp s)^*\equiv\minf$ and
$\varphi\persp s\equiv\pinf$. We therefore rule out this
trivial case henceforth.

\begin{proposition}
\label{p:302}
Let $\varphi\colon\XX\to\RXX$, let $s\colon\YY\to\RXX$, 
let $x^*\in\XX^*$ and $y^*\in\YY^*$, and suppose that 
$S=s^{-1}(\RPP)\neq\emp$. Then the following hold:
\begin{enumerate}
\item
\label{p:302i}
$(\varphi\ppersp s)^*(x^*,y^*)
=\sup_{y\in S}(\pair{y}{y^*}+s(y)\varphi^*(x^*))$.
\item
\label{p:302ii-}
Suppose that $\varphi^*(x^*)=\pm\infty$. Then
$(\varphi\ppersp s)^*(x^*,y^*)=\pm\infty$.
\item
\label{p:302iii}
Suppose that $\minf<\varphi^*(x^*)<0$. Then
$(\varphi\ppersp s)^*(x^*,y^*)=
(\haute{s}\ppersp(-\varphi^*))(y^*,x^*)$.
\item 
\label{p:302iiv}
Suppose that $\varphi^*(x^*)=0$. Then
$(\varphi\ppersp s)^*(x^*,y^*)=\sigma_{\cconv S}(y^*)$.
\item
\label{p:302ii}
Suppose that $0<\varphi^*(x^*)<\pinf$. Then
$(\varphi\ppersp s)^*(x^*,y^*)=
(\base{(-s)}\ppersp\varphi^*)(y^*,x^*)$.
\end{enumerate}
\end{proposition}
\begin{proof} 
\ref{p:302i}:
It follows from Definition~\ref{d:persp} and
Proposition~\ref{p:28}\ref{p:28i} that 
\begin{align}
\label{e:fenchel1}
(\varphi\ppersp s)^*(x^*,y^*)
&=\sup_{\substack{x\in\XX\\ y\in\YY}}
\big(\pair{x}{x^*}+\pair{y}{y^*}-(\varphi\ppersp s)(x,y)\big)
\nonumber\\
&=\sup_{\substack{x\in\XX\\ y\in S}}
\bigg(\pair{x}{x^*}+\pair{y}{y^*}-s(y)
\varphi\bigg(\frac{x}{s(y)}\bigg)\bigg)
\nonumber\\
&=\sup_{y\in S}\bigg(\pair{y}{y^*}+s(y)\bigg(\sup_{x\in\XX}
\pair{\frac{x}{s(y)}}{x^*}-
\varphi\bigg(\frac{x}{s(y)}\bigg)\bigg)\bigg)\nonumber\\
&=\sup_{y\in S}\big(\pair{y}{y^*}+s(y)\varphi^*(x^*)\big).
\end{align}

\ref{p:302ii-}: This follows from \ref{p:302i}.

\ref{p:302iii}: 
It follows from \ref{p:302i}, \eqref{e:phaut}, \eqref{e:haut},
and Lemma~\ref{l:6}\ref{l:6i} that
\begin{align}
(\varphi\ppersp s)^*(x^*,y^*)&=-\varphi^*(x^*)
\sup_{y\in S}\bigg(\pair{y}{\frac{y^*}{-\varphi^*(x^*)}}-s(y)
\bigg)\nonumber\\
&=-\varphi^*(x^*)\sup_{y\in\YY}
\bigg(\pair{y}{\frac{y^*}{-\varphi^*(x^*)}}
-\phaut{s}(y)\bigg)\nonumber\\
&=-\varphi^*(x^*)s^{{\tiny\mbox{$\wedge$}}*}\bigg(\frac{y^*}
{-\varphi^*(x^*)}\bigg)\nonumber\\
&=\big(\haute{s}\ppersp(-\varphi^*)\big)(y^*,x^*).
\end{align}

\ref{p:302iiv}: 
We derive from \ref{p:302i} and Lemma~\ref{l:99} that
$(\varphi\ppersp s)^*(x^*,y^*)=\sigma_S(y^*)=
\sigma_{\cconv S}(y^*)$.

\ref{p:302ii}: 
It follows from \ref{p:302i}, \eqref{e:pbas}, \eqref{e:bas},
and Lemma~\ref{l:6}\ref{l:6i} that
\begin{align}
(\varphi\ppersp s)^*(x^*,y^*)
&=\varphi^*(x^*)\sup_{y\in S}
\bigg(\pair{y}{\frac{y^*}{\varphi^*(x^*)}}+s(y)\bigg)\nonumber\\
&=\varphi^*(x^*)\sup_{y\in\YY}\bigg(\pair{y}{\frac{y^*}
{\varphi^*(x^*)}}-\pbas{(-s)}(y)\bigg)\nonumber\\
&=\varphi^*(x^*)(-s)^{{\tiny\mbox{$\vee$}}*}
\bigg(\frac{y^*}{\varphi^*(x^*)}\bigg)\nonumber\\
&=\big(\base{(-s)}\ppersp\varphi^*\big)(y^*,x^*),
\end{align}
as claimed.
\end{proof}

As an illustration, we consider the case of affine scaling.

\begin{example}
\label{ex:c7} 
Let $\varphi\in\Gamma_0(\XX)$, let 
$w^*\in\YY^*\smallsetminus\{0\}$, let 
$\overline{y}\in\YY$, set $s=w^*-\pair{\overline{y}}{w^*}$, 
set $S=\menge{y\in\YY}{\pair{y-\overline{y}}{w^*}>0}$, and set
$K=\menge{y\in\YY}{\pair{y}{w^*}\geq 0}$.
Let $x^*\in\XX^*$ and $y^*\in\YY^*$. If $\varphi^*(x^*)=\pinf$,
Proposition~\ref{p:302}\ref{p:302ii-} yields 
$(\varphi\ppersp s)^*(x^*,y^*)=\pinf$. Otherwise,
$\varphi^*(x^*)\in\RR$ and, since $S\neq\emp$,
it follows from Proposition~\ref{p:302}\ref{p:302i} that
\begin{align}
\label{e:5t}
(\varphi\ppersp s)^*(x^*,y^*) 
&=\sup_{y\in S}
\big(\pair{y}{y^*}+\varphi^*(x^*)\pair{y-\overline{y}}{w^*}\big)
\nonumber\\
&=\pair{\overline{y}}{y^*}+\sup_{y\in S}
\pair{y-\overline{y}}{y^*+\varphi^*(x^*)w^*}
\nonumber\\
&=\pair{\overline{y}}{y^*}+\sup_{y\in K}
\pair{y}{y^*+\varphi^*(x^*)w^*}
\nonumber\\
&=
\begin{cases}
\pair{\overline{y}}{y^*},&\text{if}\:\:\big(\exi
\beta\in\left]\minf,-\varphi^*(x^*)\right]\big)\;\;
y^*=\beta w^*;\\
\pinf,&\text{otherwise}.
\end{cases}
\end{align}
In particular, suppose that $\YY=\RR$, $w^*=1$, and 
$\overline{y}=0$, i.e.,
$s\colon y\mapsto y$. Then $\varphi\ppersp s$ is the standard
preperspective of \eqref{e:perspective2-} and \eqref{e:5t} yields
\begin{equation}
(\varphi\ppersp s)^*=\iota_{C},\quad\text{where}\quad
C=\menge{(x^*,y^*)\in\XX^*\times\RR}{\varphi^*(x^*)+y^*\leq 0},
\end{equation}
which recovers the expression given in \cite{Rock66}.
\end{example}

Next, we derive a variant of Proposition~\ref{p:302} that will be
more readily applicable.

\begin{theorem}
\label{t:7}
Let $\varphi\colon\XX\to\RX$ be proper, let $s\colon\YY\to\RXX$
be such that $S=s^{-1}(\RPP)\neq\emp$, let $x^*\in\XX^*$, and
let $y^*\in\YY^*$. Then the following hold: 
\begin{enumerate}
\item 
\label{t:7i} 
Suppose that $\varphi^*(\XX^*)\subset\RM\cup\{\pinf\}$ and 
$(\varphi^*)^{-1}(\RMM)\neq\emp$. Then 
\begin{equation} 
\label{e:37conj}
\big(\varphi\ppersp{s}\big)^{*}(x^*,y^*)=
\begin{cases}
-\varphi^*(x^*)\haute{s}\Bigg(\dfrac{y^*}{-\varphi^*(x^*)}
\Bigg),&\text{if}\:\:\minf<\varphi^*(x^*)<0;\\
\sigma_{\cconv S}(y^*),&\text{if}\:\:\varphi^*(x^*)=0;\\[3mm]
\pinf,&\text{if}\:\:\varphi^*(x^*)=\pinf.
\end{cases}
\end{equation} 
\item 
\label{t:7ii} 
Suppose that $\varphi^*(\XX^*)\subset\{0,\pinf\}$. Then 
\begin{equation} 
\label{e:38conj}
\big(\varphi
\ppersp{s}\big)^{*}(x^*,y^*)=
\iota_{(\varphi^*)^{-1}(\{0\})}(x^*)+\sigma_{\cconv S}(y^*).
\end{equation} 
\item 
\label{t:7iii} 
Suppose that $\varphi^*(\XX^*)\subset\RPX$ and 
$(\varphi^*)^{-1}(\RPP)\neq\emp$.
Then
\begin{equation}
\label{e:r7}
(\varphi\ppersp s)^*(x^*,y^*)=
\begin{cases}
\varphi^*(x^*)\base{(-s)}\Bigg(\dfrac{y^*}{\varphi^*(x^*)}\Bigg),
&\text{if}\:\:0<\varphi^*(x^*)<\pinf;\\
\sigma_{\cconv S}(y^*),&\text{if}\:\:\varphi^*(x^*)=0;\\[3mm]
\pinf,&\text{if}\:\:\varphi^*(x^*)=\pinf.
\end{cases}
\end{equation}
\item 
\label{t:7iv}
Suppose that $(\varphi^*)^{-1}(\RMM)\neq\emp$
and $(\varphi^*)^{-1}(\RPP)\neq\emp$.
Then the following hold:
\begin{enumerate}
\item
\label{t:7iva}
$(\varphi\ppersp s)^*(x^*,y^*)=
\begin{cases}
-\varphi^*(x^*)\haute{s}\Bigg(\dfrac{y^*}{-\varphi^*(x^*)}
\Bigg),&\text{if}\:\:\minf<\varphi^*(x^*)<0;\\
\sigma_{\cconv 
S}(y^*),&\text{if}\:\:\varphi^*(x^*)=0;\\
\varphi^*(x^*)\base{(-s)}\Bigg(\dfrac{y^*}{\varphi^*(x^*)}\Bigg),
&\text{if}\:\:0<\varphi^*(x^*)<\pinf;\\
\pinf,&\text{if}\:\:\varphi^*(x^*)=\pinf.
\end{cases}$
\item
\label{t:7ivb}
$(\varphi\ppersp s)^*(x^*,y^*)
=\min\big\{\big(\ebase{\varphi}\ppersp s\big)^*(x^*,y^*),
\big(\ehaute{\varphi}\ppersp s\big)^*(x^*,y^*)\big\}$.
\end{enumerate}
\end{enumerate}
\end{theorem}
\begin{proof} 
Claims \ref{t:7i}--\ref{t:7iva} follow from
Proposition~\ref{p:302}\ref{p:302ii-}--\ref{p:302ii} and
Definition~\ref{d:persp}. It remains to show \ref{t:7ivb}.
Since $\varphi$ is proper, $\minf\notin\varphi^*(\XX^*)$.
Moreover, $\dom\varphi^*\neq\emp$ and hence
$\varphi^*\in\Gamma_{0}(\XX^*)$. Therefore, applying items
\ref{l:120+}, \ref{l:12iii}, and \ref{l:12v} in 
Lemma~\ref{l:12} to $\varphi^*$ and invoking 
Lemma~\ref{l:6}\ref{l:6iv} and \ref{t:7i} yield
\begin{equation}
\label{e:peu_importe}
\big(\ebase{\varphi}\ppersp{s}\big)^{*}(x^*,y^*)=
\begin{cases}
-\varphi^{*}(x^*)\haute{s}\Bigg(\dfrac{y^*}{-\varphi^{*}(x^*)}
\Bigg),&\text{if}\:\:\minf<\varphi^*(x^*)<0;\\
\sigma_{\cconv S}(y^*),&\text{if}\:\:\varphi^*(x^*)=0;\\[3mm]
\pinf,&\text{if}\:\:0<\varphi^*(x^*)\leq\pinf.
\end{cases}
\end{equation} 
Likewise, using Lemma~\ref{l:24} and \ref{t:7iii}, we arrive at
\begin{equation}
\label{e:r75}
(\ehaute{\varphi}\ppersp s)^*(x^*,y^*)=
\begin{cases}
\varphi^*(x^*)\base{(-s)}\Bigg(\dfrac{y^*}{\varphi^*(x^*)}\Bigg),
&\text{if}\:\:0<\varphi^*(x^*)<\pinf;\\
\sigma_{\cconv S}(y^*),&\text{if}\:\: 
\minf<\varphi^{*}(x^*)\leq 0\\ 
&\quad\text{and}\:\: x^*\in\cconv(\varphi^*)^{-1}(\RPP);\\[3mm]
\pinf,&\text{otherwise.}
\end{cases}
\end{equation}
If $0\leq\varphi^{*}(x^*)\leq\pinf$, we deduce the identity from 
\ref{t:7iva}, 
\eqref{e:peu_importe}, and
\eqref{e:r75}. Now assume that $\minf<\varphi^*(x^*)<0$.
Lemma~\ref{l:24}\ref{l:24ii+} asserts that 
$\haut{s}(\dom\haut{s})\subset\RP$ and $\dom\haut{s}\subset\cconv 
S$. Hence, 
\begin{align}
\label{e:f5}
-\varphi^*(x^*)\haute{s}\bigg(\dfrac{y^*}
{-\varphi^*(x^*)}\bigg)
&=-\varphi^*(x^*)\sup_{y\in\dom\haut{s}}\bigg(\pair{y}{
\dfrac{y^*}{-\varphi^*(x^*)}}-\haut{s}(y)\bigg)
\nonumber\\
&=\sup_{y\in\cconv 
S}\big(\pair{y}{y^*}+
\varphi^*(x^*)\haut{s}(y)\big)
\nonumber\\[3mm]
&\leq\sup_{y\in\cconv 
S}\pair{y}{y^*}\nonumber\\[3mm]
&=\sigma_{\cconv S}(y^*),
\end{align}
which yields
\begin{equation}\label{e:12121}
\min\{\big(\ebase{\varphi}\ppersp
s\big)^*(x^*,y^*),\big(\ehaute{\varphi}\ppersp
s\big)^*(x^*,y^*)\}=-\varphi^*(x^*)\haute{s}\bigg(\dfrac{y^*}
{-\varphi^*(x^*)}\bigg).
\end{equation}
Thus, the conclusion follows from \ref{t:7iva}.
\end{proof}

We conclude this section by establishing conditions under which the
preperspective admits a continuous affine minorant. Note that, in
view of Lemma~\ref{l:6}\ref{l:6ii} and 
Theorem~\ref{t:7}\ref{t:7ii}, $\cam\varphi=\emp$ $\Rightarrow$
$\cam(\varphi\ppersp s)=\emp$.

\begin{corollary}
\label{c:59}
Let $\varphi\colon\XX\to\RX$ be proper and such that
$\cam\varphi\neq\emp$ and let $s\colon\YY\to\RXX$ be such that
$S=s^{-1}(\RPP)\neq\emp$. Then
\begin{equation}
\label{e:eq1c39}
\cam (\varphi\ppersp s)\neq\emp\quad\Leftrightarrow\quad
{\rm\big[\:}(\varphi^*)^{-1}(\RM)\neq\emp\quad\text{or}\quad
\cam\pbas{(-s)}\neq\emp\:{\rm\big]}.
\end{equation}
\end{corollary}
\begin{proof} 
Lemma~\ref{l:6}\ref{l:6ii} asserts
that $\cam (\varphi\ppersp s)=\emp$ 
if and only if $(\varphi\ppersp s)^{*}\equiv\pinf$. In view of
Theorem~\ref{t:7}\ref{t:7iii}, 
\begin{equation}
\label{e:eqeqdx}
\Big[\varphi^* (\XX^*)\subset\RPPX\quad\text{and}
\quad\base{(-s)}\equiv\pinf\Big]\quad
\Rightarrow\quad (\varphi\ppersp s)^{*}\equiv\pinf.
\end{equation}
An inspection of items \ref{t:7i}--\ref{t:7iva} in
Theorem~\ref{t:7} shows that the converse implication also
holds. Altogether, \eqref{e:eq1c39} follows from \eqref{e:eqeqdx}
and Lemma~\ref{l:6}\ref{l:6i}--\ref{l:6ii}. 
\end{proof}

\begin{example}
\label{ex:59}
Let $\varphi\colon\XX\to\RX$ be proper and convex, and let 
$s\colon\YY\to\RXX$ be such that $S=s^{-1}(\RPP)\neq\emp$. 
Suppose that one of the following holds:
\begin{enumerate}
\item
\label{ex:59i}
$\Gamma_0(\XX)\ni\varphi\geq\rec\varphi$.
\item
\label{ex:59ii} 
$\varphi$ is lower semicontinuous at $0$ and
$\varphi(0)\in\RPP$. 
\item 
\label{ex:59iii} 
$\cam\varphi\neq\emp$ and $-s\in\Gamma_0(\YY)$. 
\end{enumerate}
Then $\cam(\varphi\ppersp s)\neq\varnothing$. 
\end{example}
\begin{proof}
\ref{ex:59i}: 
This follows from Lemma~\ref{l:6}\ref{l:6iv}, 
Corollary~\ref{c:59}, and Lemma~\ref{l:12e}.

\ref{ex:59ii}: 
As in \cite[Proposition 13.44]{Livre1}, we have 
$\inf\varphi^*(\XX^*)=-\varphi^{**}(0)=-\varphi(0)\in\RMM$, which
yields $(\varphi^*)^{-1}(\RM)\neq\emp$. Hence the conclusion
follows from Lemma~\ref{l:6}\ref{l:6ii} and Corollary~\ref{c:59}.

\ref{ex:59iii}: According to Lemma~\ref{l:6}\ref{l:6iv}, 
$\emp\neq\cam(-s)\subset\cam\pbas{(-s)}$. Therefore, the
conclusion follows from Corollary~\ref{c:59}.
\end{proof}

\section{Perspective functions}
\label{sec:5}

We investigate the properties of the perspective function
introduced in Definition~\ref{d:persp}. We preface our analysis
with the case of affine scaling.

\begin{example}
\label{ex:307}
Let $\varphi\in\Gamma_0(\XX)$, suppose that 
$w^*\in\YY^*\smallsetminus\{0\}$, let 
$\overline{y}\in\YY$, and set $s=w^*-\pair{\overline{y}}{w^*}$.
Let $x\in\XX$ and $y\in\YY$. Then
\begin{align}
\label{e:111}
(\varphi\persp s)(x,y)
&=
\begin{cases}
\pair{y-\overline{y}}{w^*}\,\varphi\bigg(
\dfrac{x}{\pair{y-\overline{y}}{w^*}}\bigg),
&\text{if}\;\;\pair{y-\overline{y}}{w^*}>0;\\
(\rec\varphi)(x),&\text{if}\;\;\pair{y-\overline{y}}{w^*}=0;\\
\pinf,&\text{otherwise}.
\end{cases}
\end{align}
In particular, if $\YY=\RR$, $w^*=1$, and $\overline{y}=0$, we
recover the fact that $\varphi\persp s=\widetilde{\varphi}$
mentioned in Section~\ref{sec:1} (see \eqref{e:0}).
\end{example}
\begin{proof}
Since $-s\in\Gamma_0(\YY)$, it follows from 
Lemma~\ref{l:6}\ref{l:6iv} and
Example~\ref{ex:59}\ref{ex:59iii} that 
$\cam(\varphi\ppersp s)\neq\emp$. Therefore, we deduce from 
Definition~\ref{d:persp}, Lemma~\ref{l:9}\ref{l:9i}, and
Example~\ref{ex:c7} that
\begin{align}
(\varphi\persp s)(x,y)
&=(\varphi\ppersp s)^{\breve{}}(x,y)\nonumber\\
&=(\varphi\ppersp s)^{**}(x,y)\nonumber\\
&=\sup_{\substack{x^*\in\XX^*\\y^*\in\YY^*}}
\big(\pair{x}{x^*}+\pair{y}{y^*}-(\varphi\ppersp s)^*(x^*,y^*)\big)
\nonumber\\
&=\sup_{\substack{x^*\in\dom\varphi^*\\
\beta\in\left]\minf,-\varphi^*(x^*)\right]}}
\big(\pair{x}{x^*}+\beta\,\pair{y-\overline{y}}{w^*}\big)\nonumber\\
&=
\begin{cases}
\displaystyle{\sup_{x^*\in\dom\varphi^*}}
\big(\pair{x}{x^*}-\varphi^*(x^*)\pair{y-\overline{y}}{w^*}\big),
&\text{if}\;\;\pair{y-\overline{y}}{w^*}>0;\\
\sigma_{\dom\varphi^*}(x),
&\text{if}\;\;\pair{y-\overline{y}}{w^*}=0;\\
\pinf,&\text{if}\;\;\pair{y-\overline{y}}{w^*}<0,
\end{cases}
\end{align}
which, by virtue of Lemma~\ref{l:6}\ref{l:6iv} 
and Lemma~\ref{l:pjlaurent}\ref{l:pjlaurentii},
yields \eqref{e:111}.
\end{proof}

We are now ready to present our main result, which provides
explicit expressions of the perspective function in the general
case of nonlinear scaling. We state our theorem in a setting that
avoids the degenerate case when 
$(\varphi\persp s)(\XX\times\YY)\subset\{\minf,\pinf\}$.

\begin{theorem}
\label{t:55}
Let $\varphi\colon\XX\to\RX$ be proper and such 
that $\cam\varphi\neq\emp$,
let $s\colon\YY\to\RXX$ be such that $S=s^{-1}(\RPP)\neq\emp$,
and suppose that 
\begin{equation}
\label{e:55hyp}
(\varphi^*)^{-1}(\RM)\neq\emp\quad\text{or}\quad
\cam\pbas{(-s)}\neq\emp.
\end{equation}
Then
\begin{enumerate}
\item 
\label{t:55i-}
$\varphi\persp s\in\Gamma_0(\XX\oplus\YY)$.
\end{enumerate}
Furthermore, let $x\in\XX$ and $y\in\YY$. Then the following are
satisfied:
\begin{enumerate}[resume]
\item 
\label{t:55i} 
Suppose that $\varphi^*(\XX^*)\subset\RM\cup\{\pinf\}$ and 
$(\varphi^*)^{-1}(\RMM)\neq\emp$. Then 
\begin{equation} 
\label{e:37a}
\big(\varphi\persp s\big)(x,y)=
\begin{cases}
\haut{s}(y){\breve{\varphi}}
\bigg(\dfrac{x}{\haut{s}(y)}\bigg),
&\text{if}\;\;0<\haut{s}(y)<\pinf;\\
(\rec{\breve{\varphi}})(x),&
\text{if}\;\;\haut{s}(y)=0;\\[2mm]
\pinf,&\text{if}\;\;\haut{s}(y)=\pinf.
\end{cases}
\end{equation} 
\item 
\label{t:55ii} 
Suppose that $\varphi^*(\XX^*)\subset\{0,\pinf\}$. Then 
\begin{equation} 
\label{e:380}
\big(\varphi\persp s\big)(x,y)=
(\rec{\breve{\varphi}})(x)+\iota_{\cconv S}(y).
\end{equation} 
\item 
\label{t:55iii} 
Suppose that $\varphi^*(\XX^*)\subset\RPX$. Then the following 
are satisfied:
\begin{enumerate}
\item 
\label{t:55iiia} 
Suppose that $(\varphi^*)^{-1}(\{0\})\neq\emp$ and 
$\cam\pbas{(-s)}=\emp$. Then
\begin{equation}
\label{e:38a}
\big(\varphi\persp s\big)(x,y)
=\sigma_{(\varphi^*)^{-1}(\{0\})}(x)+\iota_{\cconv S}(y).
\end{equation}
\item 
\label{t:55iiib}
Suppose that $(\varphi^*)^{-1}(\RPP)\neq\emp$
and $\cam\pbas{(-s)}\neq\emp$. Then
\begin{equation}
\label{e:97a}
\big(\varphi\persp s\big)(x,y)=
\begin{cases}
-\bas{(-s)}(y){\breve{\varphi}}
\bigg(\dfrac{x}{-\bas{(-s)}(y)}\bigg),
&\text{if}\;\;\minf<\bas{(-s)}(y)<0;\\
(\rec\breve{\varphi})(x),
&\text{if}\;\;\bas{(-s)}(y)=0;\\
\pinf,&\text{if}\;\;\bas{(-s)}(y)=\pinf.
\end{cases}
\end{equation}
\end{enumerate}
\item 
\label{t:55v} 
Suppose that $(\varphi^*)^{-1}(\RMM)\neq\emp$ and that
$(\varphi^*)^{-1}(\RPP)\neq\emp$. Then the following are satisfied:
\begin{enumerate}
\item 
\label{t:55va0} 
$\big(\varphi\persp s\big)(x,y)=
\max\big\{\big(\ebase{\varphi}\persp s\big)(x,y),
\big(\ehaute{\varphi}\persp s\big)(x,y)\big\}$.
\item 
\label{t:55va} 
Suppose that
$\cam\pbas{(-s)}=\emp$. Then $\ebase{\varphi}\persp s\ge 
\ehaute{\varphi}\persp s$
and
\begin{equation} 
\label{e:37ae}
\big(\varphi\persp s\big)(x,y)=
\begin{cases}
\haut{s}(y)\ebase{\varphi}
\bigg(\dfrac{x}{\haut{s}(y)}\bigg),
&\text{if}\;\;0<\haut{s}(y)<\pinf;\\
(\rec{\ebase{\varphi}})(x),&
\text{if}\;\;\haut{s}(y)=0;\\[2mm]
\pinf,&\text{if}\;\;\haut{s}(y)=\pinf.
\end{cases}
\end{equation} 
\item 
\label{t:55vb} 
Suppose that $\cam\pbas{(-s)}\neq\emp$. Then
\begin{multline}
\label{e:cl6}
\hspace{-1.1cm}\big(\varphi\persp s\big)(x,y)=\\
\hspace{-1.1cm}\begin{cases}
\max\left\{\haut{s}(y){\ebase{\varphi}}
\bigg(\dfrac{x}{\haut{s}(y)}\bigg),-\bas{(-s)}(y){\ehaute{\varphi}}
\bigg(\dfrac{x}{-\bas{(-s)}(y)}\bigg)\right\},\:\:&\text{if}\:\:
0<\haut{s}(y)<\pinf;\\[4mm]
\max\left\{(\rec\ebase{\varphi})(x),-\bas{(-s)}(y)
{\ehaute{\varphi}}
\bigg(\dfrac{x}{-\bas{(-s)}(y)}\bigg)\right\},\:\:&\text{if}\:\: 
\bas{(-s)}(y)<0=\haut{s}(y);\\[4mm]
(\rec\breve{\varphi})(x),\:\:&\text{if}\:\: 
\bas{(-s)}(y)=0=\haut{s}(y);\\[4mm]
\pinf,&\text{if}\:\:\haut{s}(y)=\pinf,
\end{cases}
\end{multline}
where all the possible cases are exhausted.
\end{enumerate}
\end{enumerate}
\end{theorem}
\begin{proof}
Since $\cam\varphi\neq\emp$ and $\varphi\not\equiv\pinf$, by
virtue of Lemma~\ref{l:9}\ref{l:9i-}--\ref{l:9i}, we have
\begin{equation}
\label{e:dc4}
\varphi^*\in\Gamma_0(\XX^*)\quad\text{and}\quad
\varphi^{**}=\breve{\varphi}\in\Gamma_0(\XX).
\end{equation}
In turn, it follows from \eqref{e:55hyp}, Corollary~\ref{c:59}, 
Definition~\ref{d:persp}, and Lemma~\ref{l:9}\ref{l:9i} that 
\begin{equation}
\label{e:ppp}
\cam(\varphi\ppersp s)\neq\emp\quad\text{and}\quad
\varphi\persp s=(\varphi\ppersp{s})^{\breve{}}=
(\varphi\ppersp{s})^{**}.
\end{equation}
We also derive from Lemma~\ref{l:6}\ref{l:6i}, 
\eqref{e:dc4}, and 
Lemma~\ref{l:pjlaurent}\ref{l:pjlaurentii} that
\begin{equation}
\label{e:h5ys}
\sigma_{\dom\varphi^*}=\sigma_{\dom(\breve{\varphi})^*}
=\rec\breve{\varphi}
\end{equation}
and from Proposition~\ref{p:28}\ref{p:28ii} that
\begin{equation}
\label{e:s8}
\dom({\varphi}\ppersp s)\neq\emp.
\end{equation}

\ref{t:55i-}: 
This follows from \eqref{e:ppp},
\eqref{e:s8}, and Lemma~\ref{l:9}\ref{l:9i-}. 

\ref{t:55i}: Theorem~\ref{t:7}\ref{t:7i} implies that
$\dom({\varphi}\ppersp s)^*\subset 
(\varphi^*)^{-1}(\RM)\times\YY^*$. Consequently,
\begin{align}
\label{e:1144i}
({\varphi}\ppersp s)^{**}(x,y)
&=\max\Bigg\{
\sup_{\substack{x^*\in 
(\varphi^*)^{-1}(\RMM)\\y^*\in\YY^*}}
\big(\pair{x}{x^*}+\pair{y}{y^*}-
({\varphi}\ppersp s)^*(x^*,y^*)\big),\nonumber\\
&\hskip 16mm\sup_{\substack{x^*\in 
(\varphi^*)^{-1}(\{0\})\\y^*\in\YY^*}}
\big(\pair{x}{x^*}+\pair{y}{y^*}-
({\varphi}\ppersp s)^*(x^*,y^*)\big)\Bigg\}.
\end{align}
Moreover, by Theorem~\ref{t:7}\ref{t:7i}, \eqref{e:haut}, and 
Lemma~\ref{l:6}\ref{l:6i},
\begin{align}
\label{e:65362i}
\sup_{\substack{x^*\in(\varphi^*)^{-1}(\RMM)\\y^*\in\YY^*}}
\big(\pair{x}{x^*}+\pair{y}{y^*}-
({\varphi}\ppersp s)^*(x^*,y^*)\big)\nonumber\\
&\hspace{-82mm}=\sup_{\substack{x^*\in 
(\varphi^*)^{-1}(\RMM)\\y^*\in\YY^*}}
\Bigg(\pair{x}{x^*}+\pair{y}{y^*}+
\varphi^*(x^*)
\haute{s}\Bigg(\frac{y^*}
{-\varphi^*(x^*)}\Bigg)
\Bigg)\nonumber\\
&\hspace{-82mm}=\sup_{x^*\in(\varphi^*)^{-1}(\RMM)}
\Bigg(\pair{x}{x^*}-\varphi^*(x^*)\sup_{y^*\in\YY^*}
\Bigg(\pair{y}{\frac{y^*}{-\varphi^*(x^*)}}-\haute{s}
\Bigg(\frac{y^*}{-\varphi^*(x^*)}\Bigg)\Bigg)\Bigg)
\nonumber\\[3mm]
&\hspace{-82mm}
=\sup_{x^*\in(\varphi^*)^{-1}(\RMM)}
\big(\pair{x}{x^*}-\varphi^*(x^*)\haut{s}(y)\big)
\end{align}
and
\begin{align}
\label{e:0i}
\hskip -3mm
\sup_{\substack{x^*\in(\varphi^*)^{-1}(\{0\})\\y^*\in\YY^*}}
\big(\pair{x}{x^*}+\pair{y}{y^*}-
({\varphi}\ppersp s)^*(x^*,y^*)\big)
&=\sup_{\substack{x^*\in(\varphi^*)^{-1}(\{0\})\\y^*\in\YY^*}}
\big(\pair{x}{x^*}+\pair{y}{y^*}-
\sigma_{\cconv S}(y^*)\big)\nonumber\\
&=\sup_{x^*\in(\varphi^*)^{-1}(\{0\})}
\big(\pair{x}{x^*}+\iota_{\cconv S}(y)\big).
\end{align}
Hence, in view of \eqref{e:1144i} and \eqref{e:65362i},
\begin{equation}
\label{e:d71i}
({\varphi}\ppersp s)^{**}(x,y)
=\max\Bigg\{\sup_{x^*\in(\varphi^*)^{-1}(\RMM)}
\big(\pair{x}{x^*}-\varphi^{*}(x^*)\haut{s}(y)\big),
\sup_{x^*\in(\varphi^*)^{-1}(\{0\})}
\big(\pair{x}{x^*}+\iota_{\cconv S}(y)\big)\Bigg\}.
\end{equation}
In addition, Lemma~\ref{l:24}\ref{l:24ii+} yields
$\haut{s}(y)\in\RPX$. 
If $\haut{s}(y)=\pinf$, since 
$(\varphi^*)^{-1}(\RMM)\neq\emp$, then it follows from 
\eqref{e:d71i} that $({\varphi}\ppersp s)^{**}(x,y)=\pinf$. 
Now assume that $\haut{s}(y)\in\RP$. Then
Lemma~\ref{l:24}\ref{l:24ii+} yields 
\begin{equation}
\label{e:dr3i}
y\in\cconv S.
\end{equation}
Thus, if $\haut{s}(y)\in\RPP$, 
we deduce from \eqref{e:d71i}, \eqref{e:dr3i}, and \eqref{e:dc4} 
that
\begin{align}
\label{e:d7r1i}
({\varphi}\ppersp s)^{**}(x,y)
&=\sup_{x^*\in(\varphi^*)^{-1}(\RM)}
\big(\pair{x}{x^*}-\varphi^{*}(x^*)\haut{s}(y)\big)\nonumber\\
&=\sup_{x^*\in\dom\varphi^*}
\big(\pair{x}{x^*}-\varphi^{*}(x^*)\haut{s}(y)\big)\nonumber\\
&=\haut{s}(y)\breve{\varphi}\bigg(\dfrac{x}{\haut{s}(y)}\bigg).
\end{align}
Now, if $\haut{s}(y)=0$, we infer from \eqref{e:d71i} and
\eqref{e:h5ys} that 
\begin{equation}
(\varphi\ppersp s)^{**}(x,y)=
\max\big\{\sigma_{(\varphi^*)^{-1}(\RMM)}(x),
\sigma_{(\varphi^*)^{-1}(\{0\})}(x)\big\}
=\sigma_{\dom\varphi^*}(x)=\big(\rec\breve{\varphi}\big)(x).
\end{equation}
Hence, \eqref{e:37a} holds.

\ref{t:55ii}: Theorem~\ref{t:7}\ref{t:7ii} implies that
$\emp\neq\dom({\varphi}\ppersp s)^*\subset 
(\varphi^*)^{-1}(\{0\})\times\YY^*$. Hence, we have
\begin{align}
({\varphi}\ppersp s)^{**}(x,y)
&=\sup_{\substack{x^*\in 
(\varphi^*)^{-1}(\{0\})\\y^*\in\YY^*}}
\big(\pair{x}{x^*}+\pair{y}{y^*}-
({\varphi}\ppersp s)^*(x^*,y^*)\big)\nonumber\\
&=\sup_{\substack{x^*\in
(\varphi^*)^{-1}(\{0\})\\y^*\in\YY^*}}
\big(\pair{x}{x^*}+\pair{y}{y^*}-
\sigma_{\cconv S}(y^*)\big)\nonumber\\
&=\sup_{x^*\in\dom\varphi^*}
\big(\pair{x}{x^*}+\iota_{\cconv S}(y)\big)\nonumber\\
&=\sigma_{\dom\varphi^*}(x)+
\iota_{\cconv S}(y),
\end{align}
and we obtain \eqref{e:380} from \eqref{e:h5ys}. 

\ref{t:55iii}:
Theorem~\ref{t:7}\ref{t:7iii} implies that
$\dom({\varphi}\ppersp s)^*\subset 
(\varphi^*)^{-1}(\RP)\times\YY^*$, which yields
\begin{align}
\label{e:1144iii}
({\varphi}\ppersp s)^{**}(x,y)
&=\max\Bigg\{
\sup_{\substack{x^*\in 
(\varphi^*)^{-1}(\{0\})\\y^*\in\YY^*}}
\big(\pair{x}{x^*}+\pair{y}{y^*}-
({\varphi}\ppersp s)^*(x^*,y^*)\big),\nonumber\\
&\hskip 16mm\sup_{\substack{x^*\in 
(\varphi^*)^{-1}(\RPP)\\y^*\in\YY^*}}
\big(\pair{x}{x^*}+\pair{y}{y^*}-
({\varphi}\ppersp s)^*(x^*,y^*)\big)\Bigg\}.
\end{align}
Moreover, by Theorem~\ref{t:7}\ref{t:7iii}, 
as in \eqref{e:0i},
\begin{equation}
\label{e:0iii}
\hspace{-3mm}\sup_{\substack{x^*\in(\varphi^*)^{-1}
(\{0\})\\y^*\in\YY^*}}
\big(\pair{x}{x^*}+\pair{y}{y^*}-({\varphi}\ppersp s)^*(x^*,y^*)
\big)
=\sup_{x^*\in(\varphi^*)^{-1}(\{0\})}
\big(\pair{x}{x^*}+\iota_{\cconv S}(y)\big)
\end{equation}
and, using \eqref{e:bas} and Lemma~\ref{l:6}\ref{l:6i},
\begin{align}
\label{e:65362iii}
\sup_{\substack{x^*\in(\varphi^*)^{-1}(\RPP)\\y^*\in\YY^*}}
\big(\pair{x}{x^*}+\pair{y}{y^*}-
({\varphi}\ppersp s)^*(x^*,y^*)\big)\nonumber\\
&\hspace{-82mm}=\sup_{\substack{x^*\in 
(\varphi^*)^{-1}(\RPP)\\y^*\in\YY^*}}
\Bigg(\pair{x}{x^*}+\pair{y}{y^*}-
\varphi^*(x^*)
\base{(-s)}\Bigg(\frac{y^*}
{\varphi^*(x^*)}\Bigg)
\Bigg)\nonumber\\
&\hspace{-82mm}=\sup_{x^*\in(\varphi^*)^{-1}(\RPP)}
\Bigg(\pair{x}{x^*}+\varphi^*(x^*)\sup_{y^*\in\YY^*}
\Bigg(\pair{y}{\frac{y^*}{\varphi^*(x^*)}}-\base{(-s)}
\Bigg(\frac{y^*}{\varphi^*(x^*)}\Bigg)\Bigg)\Bigg)
\nonumber\\[3mm]
&\hspace{-82mm}
=\sup_{x^*\in(\varphi^*)^{-1}(\RPP)}
\big(\pair{x}{x^*}+\varphi^*(x^*)\bas{(-s)}(y)\big).
\end{align}
Combining \eqref{e:1144iii}, \eqref{e:0iii}, and 
\eqref{e:65362iii}, we get
\begin{multline}
\label{e:d71}
({\varphi}\ppersp s)^{**}(x,y)
=\\
\hskip -2mm\max\Bigg\{
\sup_{x^*\in(\varphi^*)^{-1}(\{0\})}
\big(\pair{x}{x^*}+\iota_{\cconv S}(y)\big),
\sup_{x^*\in(\varphi^*)^{-1}(\RPP)}
\big(\pair{x}{x^*}+\varphi^{*}(x^*)\bas{(-s)}(y)\big)
\Bigg\}.
\end{multline}

\ref{t:55iiia}: Lemma~\ref{l:6}\ref{l:6ii} asserts that
$\bas{(-s)}\equiv\minf$. Therefore, since
$(\varphi^*)^{-1}(\{0\})\neq\emp$, we deduce from \eqref{e:d71} 
that
\begin{equation}
({\varphi}\ppersp s)^{**}(x,y)
=\sup_{x^*\in(\varphi^*)^{-1}(\{0\})}
\bigl(\pair{x}{x^*}+\iota_{\cconv S}(y)\bigr)
=\sigma_{(\varphi^*)^{-1}(\{0\})}(x)+\iota_{\cconv S}(y),
\end{equation}
as announced in \eqref{e:38a}.

\ref{t:55iiib}: 
Lemma~\ref{l:12}\ref{l:120+} yields
$\bas{(-s)}(y)\in\RM\cup\{\pinf\}$. 
If $\bas{(-s)}(y)=\pinf$, since 
$(\varphi^*)^{-1}(\RPP)\neq\emp$, it follows from 
\eqref{e:d71} that $({\varphi}\ppersp s)^{**}(x,y)=\pinf$. 
Now assume that $\minf<\bas{(-s)}(y)\leq 0$. Then
Lemma~\ref{l:12}\ref{l:120+} yields 
\begin{equation}
\label{e:dr3}
y\in\cconv S.
\end{equation}
Thus, if $\bas{(-s)}(y)=0$, we infer from \eqref{e:d71} and
\eqref{e:h5ys} that 
\begin{equation}
(\varphi\ppersp s)^{**}(x,y)=
\max\big\{\sigma_{(\varphi^*)^{-1}(\{0\})}(x),
\sigma_{(\varphi^*)^{-1}(\RPP)}(x)\big\}
=\sigma_{\dom\varphi^*}(x)=\big(\rec\breve{\varphi}\big)(x).
\end{equation}
Next, assume that $\bas{(-s)}(y)<0$. Then we deduce from 
\eqref{e:dr3}, \eqref{e:d71}, and \eqref{e:dc4} that
\begin{align}
\label{e:d7r1}
({\varphi}\ppersp s)^{**}(x,y)
&=\sup_{x^*\in(\varphi^*)^{-1}(\RP)}
\big(\pair{x}{x^*}+\varphi^{*}(x^*)\bas{(-s)}(y)\big)\nonumber\\
&=\sup_{x^*\in\dom\varphi^*}
\big(\pair{x}{x^*}+\varphi^{*}(x^*)\bas{(-s)}(y)\big)\nonumber\\
&=-\bas{(-s)}(y)
\breve{\varphi}\bigg(\dfrac{x}{-\bas{(-s)}(y)}\bigg).
\end{align}
This verifies that \eqref{e:97a} holds.

\ref{t:55v}: 
We deduce from Lemma~\ref{l:12}\ref{l:120+} and
Lemma~\ref{l:24}\ref{l:24ii} that
$\ebas{\varphi}\in\Gamma_0(\XX^*)$ and
$\ehaut{\varphi}\in\Gamma_0(\XX^*)$. In turn, 
Lemma~\ref{l:6}\ref{l:6iv} yields
\begin{equation}
\label{e:cl8}
(\ebase{\varphi})^{\breve{}}=
\ebase{\varphi}\in\Gamma_0(\XX)
\quad\text{and}\quad
(\ehaute{\varphi})^{\breve{}}=
\ehaute{\varphi}\in\Gamma_0(\XX). 
\end{equation}
Note also that \eqref{e:dc4}, Lemma~\ref{l:12}\ref{l:12iii}, and
Lemma~\ref{l:24}\ref{l:24v} imply that
\begin{equation}
\label{e:cedweq}
(\ebas{\varphi})^{-1}(\RMM)=
(\varphi^*)^{-1}(\RMM)\neq\emp\quad\text{and}\quad
(\ehaut{\varphi})^{-1}(\{0\})\neq\emp.
\end{equation}
We derive from Corollary~\ref{c:59}, Lemma~\ref{l:6}\ref{l:6i}, 
and \eqref{e:cedweq} that 
$\cam (\ebase{\varphi}\ppersp s)\neq\emp$ and 
$\cam(\ehaute{\varphi}\ppersp s)\neq\emp$. 
Therefore we deduce from Lemma~\ref{l:9}\ref{l:9i} that
\begin{equation}
\label{e:2121iii}
\ebase{\varphi}\persp s=\big(\ebase{\varphi}\ppersp s\big)^{**}
\quad\text{and}\quad\ehaute{\varphi}\persp s=
\big(\ehaute{\varphi}\ppersp s\big)^{**}.
\end{equation}

\ref{t:55va0}: It follows from Theorem~\ref{t:7}\ref{t:7ivb} that
\begin{equation}
\label{e:851a}
\big(\varphi\ppersp s\big)^{**}(x,y)=
\max\Big\{\big(\ebase{\varphi}\ppersp s\big)^{**}(x,y),
\big(\ehaute{\varphi}\ppersp s\big)^{**}(x,y)\Big\}.
\end{equation}
Thus, the claim follows from \eqref{e:ppp} and \eqref{e:2121iii}.

\ref{t:55va}: According to \eqref{e:dc4} and
Lemma~\ref{l:12}\ref{l:12ii+},
$\dom\ebas{\varphi}=(\varphi^*)^{-1}(\RM)$. Hence, using
Theorem~\ref{t:7}\ref{t:7i}, Lemma~\ref{l:6}\ref{l:6i}, and 
\eqref{e:cedweq}, we arrive at 
$\dom (\ebase{\varphi}\ppersp s)^*\subset\dom 
\ebas{\varphi}\times\YY^*=(\varphi^*)^{-1}(\RM)\times\YY^*$. 
Therefore, it follows from \eqref{e:2121iii} that
\begin{align}
\label{e:1144i211}
\big(\ebase{\varphi}\persp s\big)(x,y)
&=\max\Bigg\{
\sup_{\substack{x^*\in (\varphi^*)^{-1}(\RMM)\\y^*\in\YY^*}}
\big(\pair{x}{x^*}+\pair{y}{y^*}-
(\ebase{\varphi}\ppersp s)^*(x^*,y^*)\big),\nonumber\\
&\hskip 16mm\sup_{\substack{x^*\in 
(\varphi^*)^{-1}(\{0\})\\y^*\in\YY^*}}
\big(\pair{x}{x^*}+\pair{y}{y^*}-
(\ebase{\varphi}\ppersp s)^*(x^*,y^*)\big)\Bigg\}.
\end{align}
On the one hand, Theorem~\ref{t:7}\ref{t:7i} applied to 
$\ebase{\varphi}$ and $s$, Lemma~\ref{l:12}\ref{l:12iii} 
applied to $\varphi^*$, Lemma~\ref{l:6}\ref{l:6i}, 
and Lemma~\ref{l:24}\ref{l:24ii+} applied to $s$ 
yield 
\begin{align}
\label{e:65362iiibis}
&\sup_{\substack{x^*\in(\varphi^*)^{-1}(\RMM)\\y^*\in\YY^*}}
\big(\pair{x}{x^*}+\pair{y}{y^*}-
(\ebase{\varphi}\ppersp s)^*(x^*,y^*)\big)\nonumber\\
&\hspace{20mm}=\sup_{\substack{x^*\in 
(\varphi^*)^{-1}(\RMM)\\y^*\in\YY^*}}
\Bigg(\pair{x}{x^*}+\pair{y}{y^*}+
\ebas{\varphi}(x^*)
\haute{s}\Bigg(\frac{y^*}
{-\ebas{\varphi}(x^*)}\Bigg)
\Bigg)\nonumber\\
&\hspace{20mm}=\sup_{x^*\in(\varphi^*)^{-1}(\RMM)}
\Bigg(\pair{x}{x^*}-\varphi^*(x^*)\sup_{y^*\in\YY^*}
\Bigg(\pair{y}{\frac{y^*}{-\varphi^*(x^*)}}-\haute{s}
\Bigg(\frac{y^*}{\varphi^*(x^*)}\Bigg)\Bigg)\Bigg)
\nonumber\\[3mm]
&\hspace{20mm}
=\sup_{x^*\in(\varphi^*)^{-1}(\RMM)}
\big(\pair{x}{x^*}-\varphi^*(x^*)\haut{s}(y)\big)\nonumber\\[3mm]
&\hspace{20mm}\geq\sup_{x^*\in(\varphi^*)^{-1}(\RMM)}
\big(\pair{x}{x^*}+\iota_{\cconv S}(y)\big).
\end{align}
On the other hand, with the help of Lemma~\ref{l:12}\ref{l:12v},
Theorem~\ref{t:7}\ref{t:7i} applied to $\ebase{\varphi}$ implies
that
\begin{align}
\label{e:11e}
&\sup_{\substack{x^*\in 
(\varphi^*)^{-1}(\{0\})\\y^*\in\YY^*}}
\big(\pair{x}{x^*}+\pair{y}{y^*}-
(\ebase{\varphi}\ppersp s)^*(x^*,y^*)\big)\nonumber\\
&\hspace{20mm}=\sup_{\substack{x^*\in 
(\varphi^*)^{-1}(\{0\})\\y^*\in\YY^*}}
\big(\pair{x}{x^*}+\pair{y}{y^*}-
\sigma_{\cconv S}(y^*)\big)\nonumber\\
&\hspace{20mm}=\sup_{\substack{x^*\in (\varphi^*)^{-1}(\{0\})}}
\big(\pair{x}{x^*}+\iota_{\cconv S}(y)\big).
\end{align}
Combining \eqref{e:1144i211}, \eqref{e:65362iiibis}, 
\eqref{e:11e}, Lemma~\ref{l:24}\ref{l:24ii++}, \eqref{e:cedweq},
and \ref{t:55iiia} we obtain 
\begin{align}
\big(\ebase{\varphi}\persp s\big)(x,y)&\geq 
\sigma_{(\varphi^*)^{-1}(\RM)}(x)+\iota_{\cconv S}(y)\nonumber\\
&\geq 
\sigma_{(\ehaut{\varphi})^{-1}(\{0\})}(x)+
\iota_{\cconv S}(y)\nonumber\\
&=\big(\ehaute{\varphi}\persp s\big)(x,y). 
\end{align}
Altogether, the result follows from \eqref{e:cedweq},
\ref{t:55va0}, and \ref{t:55i} applied to $\ebase{\varphi}$ and
$s$.

\ref{t:55vb}: 
Using Lemma~\ref{l:24}\ref{l:24ii+} and Lemma~\ref{l:14}, we 
partition $\YY$ as
\begin{multline}
\label{e:cl3}
\YY=
(\haut{s})^{-1}(\RPP)\bigcup\,\Big((\haut{s})^{-1}(\{0\})
\cap\big(\bas{(-s)})^{-1}(\RMM\big)\Big)\\
\bigcup\,
\Big((\haut{s})^{-1}(\{0\})\cap\big(\bas{(-s)})^{-1}(\{0\}\big)
\Big)
\bigcup\,(\haut{s})^{-1}(\{\pinf\}),
\end{multline}
which corresponds to the cases in \eqref{e:cl6}. Therefore, it
follows from \eqref{e:cl8},
Lemma~\ref{l:pjlaurent}\ref{l:pjlaurentii},
Lemma~\ref{l:12}\ref{l:12ii+}, Lemma~\ref{l:24}\ref{l:24iv},
Lemma~\ref{l:6}\ref{l:6i}, and 
\eqref{e:h5ys} that
\begin{align}
\label{e:recmax}
\max\left\{(\rec\ebase{\varphi})(x),
(\rec\ehaute{\varphi})(x)\right\}
&=\max\left\{\sigma_{\dom\ebas{\varphi}}(x),
\sigma_{\dom\ehaut{\varphi}}(x)\right\}\nonumber\\
&=\sigma_{\dom{\varphi}^*}(x)\nonumber\\
&=(\rec\breve{\varphi})(x).
\end{align}
Altogether, \eqref{e:cl6} follows from \ref{t:55va0}, by applying
\ref{t:55i} to $\ebase{\varphi}\persp s$ and \ref{t:55iiib} to
$\ehaute{\varphi}\persp s$, and invoking \eqref{e:cl3},
\eqref{e:cl8}, and \eqref{e:recmax}.
\end{proof}

Next, we focus on the case when $\varphi\in\Gamma_0(\XX)$ and 
$\pm s\in\Gamma_0(\YY)$. We express the results in terms of
recession functions via Lemma~\ref{l:12e}, which does not involve
the sign of $\varphi^*$.
 
\begin{corollary}
\label{c:305}
Let $\varphi\in\Gamma_0(\XX)$ and let
$s\colon\YY\to\RXX$ be such that $S=s^{-1}(\RPP)\neq\emp$.
Let $x\in\XX$ and $y\in\YY$. Then the following hold: 
\begin{enumerate}
\item
\label{c:305i} 
Suppose that $\varphi\geq\rec\varphi\neq\varphi$ and
$s\in\Gamma_0(\YY)$. Then
\begin{equation} 
\label{e:365}
(\varphi\persp s)(x,y)=
\begin{cases}
s(y){\varphi}
\bigg(\dfrac{x}{{s}(y)}\bigg),
&\text{if}\;\;0<s(y)<\pinf;\\
(\rec\varphi)(x),&\text{if}\;\;y\in\cconv 
S\:\:\text{and}\:\:s(y)\leq 0;\\
\pinf,&\text{otherwise}.
\end{cases}
\end{equation} 
\item
\label{c:305ii} 
Suppose that $\varphi=\rec\varphi$. Then
$(\varphi\persp s)(x,y)=\varphi(x)+\iota_{\cconv S}(y)$.
\item
\label{c:305iii} 
Suppose that $\varphi\neq\rec\varphi$, $\varphi(0)\leq 0$, 
and $-s\in\Gamma_0(\YY)$. Then 
\begin{equation}
\label{e:975}
(\varphi\persp s)(x,y)=
\begin{cases}
{s}(y){\varphi}\bigg(\dfrac{x}{s(y)}\bigg),
&\text{if}\;\;0<s(y)<\pinf;\\
(\rec\varphi)(x),&\text{if}\;\;s(y)=0;\\
\pinf,&\text{otherwise}.
\end{cases}
\end{equation}
\end{enumerate}
Furthermore, in each case, 
$\varphi\persp s\in\Gamma_0(\XX\oplus\YY)$. 
\end{corollary}
\begin{proof} 
We first observe that Lemma~\ref{l:6}\ref{l:6iv} yields
$\breve{\varphi}=\varphi$. Furthermore, by 
Lemma~\ref{l:pjlaurent}\ref{l:pjlaurentiv}, we have
\begin{equation}
\label{e:g6}
\varphi=\rec\varphi\quad\Leftrightarrow\quad 
\varphi^*(\XX^*)\subset\{0,\pinf\}.
\end{equation}

\ref{c:305i}: 
Lemma~\ref{l:12e} and \eqref{e:g6} yield
\begin{equation}
\label{e:k4}
\dom\varphi^*=(\varphi^*)^{-1}(\RM) \quad 
\text{and} \quad (\varphi^*)^{-1}(\RMM)\neq \emp.
\end{equation}
Hence, \eqref{e:365} follows from Theorem~\ref{t:55}\ref{t:55i} and
Lemma~\ref{l:24} applied to $s$. 

\ref{c:305ii}: 
This assertion follows from \eqref{e:g6} and 
Theorem~\ref{t:55}\ref{t:55ii}.

\ref{c:305iii}: 
We have $(\forall x^*\in\XX^*)$ $\varphi^{*}(x^*)\geq
\pair{0}{x^*}-\varphi(0)\geq 0$. Thus, \eqref{e:g6} yields
\begin{equation}
\label{e:k6}
\dom\varphi^*=(\varphi^*)^{-1}(\RP) 
\quad \text{and} \quad (\varphi^*)^{-1}(\RPP)\neq \emp.
\end{equation}
Thus, since $\cam\pbas{(-s)}\neq\emp$ by 
Lemma~\ref{l:6}\ref{l:6iv}, \eqref{e:975} follows from
Theorem~\ref{t:55}\ref{t:55iiib} and Lemma~\ref{l:12} applied to
$-s$. 

Finally, since \eqref{e:55hyp} holds in each case, we deduce from
Theorem~\ref{t:55}\ref{t:55i-} that 
$\varphi\persp s\in\Gamma_0(\XX\oplus\YY)$.
\end{proof}

\begin{remark}
\label{r:z}
As mentioned in the Introduction, in the context of
Corollary~\ref{c:305}, alternative notions of
perspective functions with nonlinear scaling were proposed in
\cite{Mar05a,Zali08} under additional restrictions on the scaling
function. Specically, these papers deal with 
operations {\footnotesize$\Delta_{\text{1}}$}
and {\footnotesize$\Delta_{\text{2}}$} between functions 
$\varphi\in\Gamma_0(\XX)$ and $\psi\in\Gamma_0(\YY)$ in the
following scenarios.
\begin{enumerate}
\item
\label{r:zi}
Suppose that $\varphi\geq\rec\varphi\neq\varphi$
and $\psi(\dom\psi)\subset\RP$. In view of \eqref{e:0},
\begin{equation} 
\label{e:tI}
\varphi{\,\mbox{\footnotesize$\Delta_{\text{2}}$}\,}\psi
\colon\XX\oplus\YY\to\RX
\colon (x,y)\mapsto 
\begin{cases}
\widetilde{\varphi}\big(x,\psi(y)\big)&\text{if}\:\:y
\in\dom \psi;\\
\pinf,&\text{if}\:\:y\notin\dom \psi.
\end{cases}
\end{equation} 
Now suppose that $\psi^{-1}(\RPP)\neq\emp$.
It follows from Corollary~\ref{c:305}\ref{c:305i} that
\begin{equation}
\label{e:cz4}
\varphi{\,\mbox{\footnotesize$\Delta_{\text{2}}$}\,}\psi\leq
\varphi\persp \psi\colon (x,y)\mapsto
(\varphi{\mbox{\,\footnotesize$\Delta_{\text{2}}$\,}}\psi)(x,y)
+\iota_{\cconv \psi^{-1}(\RPP)}(y).
\end{equation}
Let us note that, since equality fails above, the 
$\varphi{\,\mbox{\footnotesize$\Delta_{\text{2}}$}\,}\psi $ is not
the largest minorant of $\varphi\ppersp \psi$ in
$\Gamma_0(\XX\oplus\YY)$. For instance, suppose that 
\begin{equation}
\label{e:cz5}
\YY=\RR\quad\text{and}\quad\psi\colon y\mapsto\max\{0,y\}. 
\end{equation}
Then $\cconv\psi^{-1}(\RPP)=\RP$ and therefore,
if $y\in\RMM$ and $0\in\dom\varphi$, we have $\psi(y)=0$ and
$0=(\varphi{\,\mbox{\footnotesize$\Delta_{\text{2}}$}\,}\psi)(0,y)
<(\varphi\persp \psi)(0,y)=\pinf$.
\item
\label{r:ziii}
Suppose that $\varphi\neq\rec\varphi$, $\varphi(0)\leq 0$, 
and $\psi(\dom\psi)\subset\RM$. In 
view of \eqref{e:0}, 
\begin{equation}
\label{e:tII}
\varphi{\,\mbox{\footnotesize$\Delta_{\text{1}}$}\,}\psi
\colon\XX\oplus\YY\to\RX\colon (x,y)\mapsto 
\begin{cases}
\widetilde{\varphi}\big(x,-\psi(y)\big),&\text{if}\:\:
y\in\dom \psi;\\ 
\pinf,&\text{if}\:\:y\notin\dom \psi.
\end{cases}
\end{equation} 
Now suppose that $\psi^{-1}(\RMM)\neq\emp$. Then it follows from 
Corollary~\ref{c:305}\ref{c:305iii} that 
$\varphi{\mbox{\,\footnotesize$\Delta_{\text{1}}$\,}}
\psi=\varphi\persp (-\psi)$. In turn,
Definition~\ref{d:persp} asserts that, in this particular scenario,
$\varphi{\mbox{\,\footnotesize$\Delta_{\text{1}}$\,}}\psi$ is the 
largest minorant of $\varphi\ppersp (-\psi)$ in
$\Gamma_0(\XX\oplus\YY)$.
\end{enumerate}
The construction proposed in
Definition~\ref{d:persp} covers a much broader range of functions
$(\varphi,s)$ that those employed above. Concrete instances will be
presented in Section~\ref{sec:6}.
\end{remark}

\begin{remark}
\label{r:55} 
Let $\varphi\in\Gamma_0(\XX)$ and let $s\colon\YY\to\RXX$ be such
that $s^{-1}(\RPP)\neq\emp$. The above remark reveals some
particular instances in which $\varphi\persp s$ can be expressed in
terms of the classical perspective of \eqref{e:0} applied to
certain transformations of $\varphi$ and $s$. Let us clarify these
identities and, in particular, address the natural question that
arises as to the validity of the identity 
\begin{equation} 
\label{e:t7}
\big(\varphi\persp s\big)(x,y)=
\begin{cases}
\widetilde{\varphi}\big(x,s(y)\big)
&\text{if}\:\:s(y)\in\RR;\\
\pinf,&\text{otherwise}
\end{cases}
\end{equation} 
beyond the classical case already discussed in Section~\ref{sec:1}
in which $\YY=\RR$ and $s\colon y\mapsto y$. It turns out that
\eqref{e:t7} is true only in very specific instances, some of which
are provided below. Let $x\in\XX$ and $y\in\YY$. Then it follows
from Theorem~\ref{t:55} that the following hold: 
\begin{enumerate}
\item 
\label{r:55i} 
Suppose that $\varphi^*(\XX^*)\subset\RM\cup\{\pinf\}$. Then 
\begin{equation} 
\label{e:37ar}
\big(\varphi\persp s\big)(x,y)=
\begin{cases}
\widetilde{{\varphi}}\big(x,\haut{s}(y)\big),
&\text{if}\:\:y\in\dom\haut{s};\\
\pinf,&\text{if}\:\: y\notin\dom\haut{s}.
\end{cases}
\end{equation} 
If we assume additionally that $s=\haut{s}$, then
it follows from Lemma~\ref{l:12e} that \eqref{e:t7} holds. This
corresponds to the setting of Remark~\ref{r:z}\ref{r:zi}.
\item 
\label{r:55ii} 
Suppose that $\varphi^*(\XX^*)\subset\RPX$, 
$(\varphi^*)^{-1}(\RPP)\neq\emp$,
and $\cam\pbas{(-s)}\neq\emp$. Then
\begin{align}
\label{e:97ar}
\big(\varphi\persp s\big)(x,y)&=
\begin{cases}
\widetilde{{\varphi}}\big(x,-\bas{(-s)}(y)\big),
&\text{if}\:\: y\in\dom\bas{(-s)};\\
\pinf,&\text{if}\:\: y\notin\dom\bas{(-s)}.
\end{cases}
\end{align}
If we assume additionally that 
$s=-\bas{(-s)}$,
then \eqref{e:t7} holds. This corresponds to the setting of
Remark~\ref{r:z}\ref{r:ziii}.
\item 
Suppose that $w^*\in\YY^*\smallsetminus\{0\}$, 
$\overline{y}\in\YY$, and $s=w^*-\pair{\overline{y}}{w^*}$. Then
Example~\ref{ex:307} implies that \eqref{e:t7} holds. 
\item 
\label{r:55iii} 
Suppose that $(\varphi^*)^{-1}(\RMM)\neq\emp$ and that
$(\varphi^*)^{-1}(\RPP)\neq\emp$. Then 
\begin{multline}
\label{e:cl6r}
\hspace{-1.1cm}\big(\varphi\persp s\big)(x,y)=
\begin{cases}
\max\big\{\widetilde{\ebase{\varphi}}\big(x,\haut{s}(y)\big),
\widetilde{\ehaute{\varphi}}\big(x,-\bas{(-s)}(y)\big)\big\},
&\text{if}\:\:y\in\dom\haut{s};\\
\pinf,&\text{if}\:\:y\notin\dom\haut{s}.
\end{cases}
\end{multline}
If $s=\haut{s}=-\bas{(-s)}$, it follows
from Remark~\ref{r:0} that \eqref{e:t7} holds.
\end{enumerate}
\end{remark}

\section{Examples and applications}
\label{sec:6}

We illustrate various cases that arise in Theorem~\ref{t:55}.

\begin{example}
\label{ex:z1}
Suppose that $\XX$ is a nonzero real reflexive Banach space, let 
$\alpha\in\RPP$, let $p\in\left]1,\pinf\right[$, set 
$\petoile=p/(p-1)$, and set 
\begin{equation}
\label{e:def_varphi_1}
\varphi_1\colon\XX\to\RR\colon x\mapsto 
\begin{cases}
\alpha\|x\|,&\;\;\text{if}\;\;\|x\|>\alpha^{\frac{1}{p-1}};\\[2mm]
\dfrac{\|x\|^p}{p}+\dfrac{\alpha^{\petoile}}{\petoile},&\;\;
\text{if}\;\;\|x\|\leq\alpha^{\frac{1}{p-1}}.
\end{cases}
\end{equation}
Suppose that $\YY=\RR$, let $\beta\in\lzeroun$, and set
\begin{equation}
\label{e:s1}
s\colon\RR\to\RX\colon y\mapsto 
\begin{cases}
y-\dfrac{\beta^2+1}{2},&\;\;\text{if}\;\;y>1;\\[2mm]
\dfrac{|y|^2-\beta^2}{2},
&\;\;\text{if}\;-1\leq y\leq 1;\\
\pinf,&\;\;\text{if}\;\;y<-1.
\end{cases}
\end{equation}
It follows from Example~\ref{ex:normp} that
$\varphi_1^*(\XX^*)\subset\RM\cup\{\pinf\}$ and 
$(\varphi_1^*)^{-1}(\RMM)\neq\emp$. Furthermore,
Lemma~\ref{l:24}\ref{l:24ii} yields
\begin{equation}
\label{e:defsextriangle}
\haut{s}\colon y\mapsto 
\begin{cases}
y-\dfrac{\beta^2+1}{2},&\;\;\text{if}\;\;y>1;\\[3mm]
\dfrac{|y|^2-\beta^2}{2},&\;\;\text{if}\;\;
-1<y\leq -\beta \:\:\text{or}\:\:\beta<y\leq 1;\\[1mm]
0,&\;\;\text{if}\;\;-\beta\leq y\leq \beta;\\
\pinf,&\;\;\text{if}\;\;\,y<-1.
\end{cases}
\end{equation}
We thus derive $\varphi_1\persp s$ from
Theorem~\ref{t:55}\ref{t:55i}; see Figure~\ref{fig:exz}.
\end{example}

\begin{example}
\label{ex:z2}
Suppose that $\XX$ is a nonzero real
reflexive Banach space, let $\alpha\in\RPP$, let
$p\in\left]1,\pinf\right[$, set $\petoile=p/(p-1)$, and set 
\begin{equation}
\label{e:def_varphi_2}
\varphi_2\colon\XX\to\RR\colon x\mapsto 
\begin{cases}
\dfrac{\|x\|^p}{p}+\dfrac{\alpha^{\petoile}}{\petoile},&
\text{if}\;\;\|x\|>\alpha^{\frac{1}{p-1}};\\[5pt]
\alpha\|x\|,&\text{if}\;\;\|x\|\leq \alpha^{\frac{1}{p-1}}.
\end{cases}
\end{equation}
Let $\YY$ and $s$ be as in Example~\ref{ex:z1}. 
In view of Example~\ref{ex:normp}, we have 
$\varphi_2^*(\XX^*)\subset\RP$ 
and $(\varphi_2^*)^{-1}(\RPP)\neq\emp$. 
Additionally,
$\cam\pbas{(-s)}\neq\emp$ and
\eqref{e:bas} yields
\begin{equation}
\label{e:defsextrianglebas}
-\bas{(-s)}\colon y\mapsto 
\begin{cases}
y+\dfrac{3-\beta^2}{2},&\;\;\text{if}\;\; y\geq -1;\\
\minf,&\:\:\text{if}\;\;y<-1.
\end{cases}
\end{equation}
We thus derive $\varphi_2\persp s$ from
Theorem~\ref{t:55}\ref{t:55iiib}; see Figure~\ref{fig:exz}.
\end{example}

\begin{example}
\label{ex:z3}
Suppose that $\XX$ is a nonzero real reflexive Banach space, let 
$\alpha\in\RPP$, let $p\in\left]1,\pinf\right[$, set 
$\petoile=p/(p-1)$, and set 
\begin{equation}
\label{e:def_varphi_3}
\varphi_3\colon\XX\to\RR\colon x\mapsto 
{\|x\|^p}/{p}+{\alpha^{\petoile}}/{\petoile}.
\end{equation}
Let $\YY$ and $s$ be as in Example~\ref{ex:z1}. 
Then, as seen in Example~\ref{ex:normp}, 
$(\varphi_3^*)^{-1}(\RMM)\neq\emp$, 
$(\varphi_3^*)^{-1}(\RPP)\neq\emp$, and it follows from
\eqref{e:r8}, \eqref{e:def_varphi_1}, and \eqref{e:def_varphi_2}
that $\ebase{\varphi_3}=\varphi_1$ and
$\ehaute{\varphi_3}=\varphi_2$. 
Hence, we derive $\varphi_3\persp s$ from
Theorem~\ref{t:55}\ref{t:55va0}; see Figure~\ref{fig:exz}.
\end{example}

\begin{figure}
\begin{subfigure}{.5\textwidth}
  \centering
  \includegraphics[width=.8\linewidth]{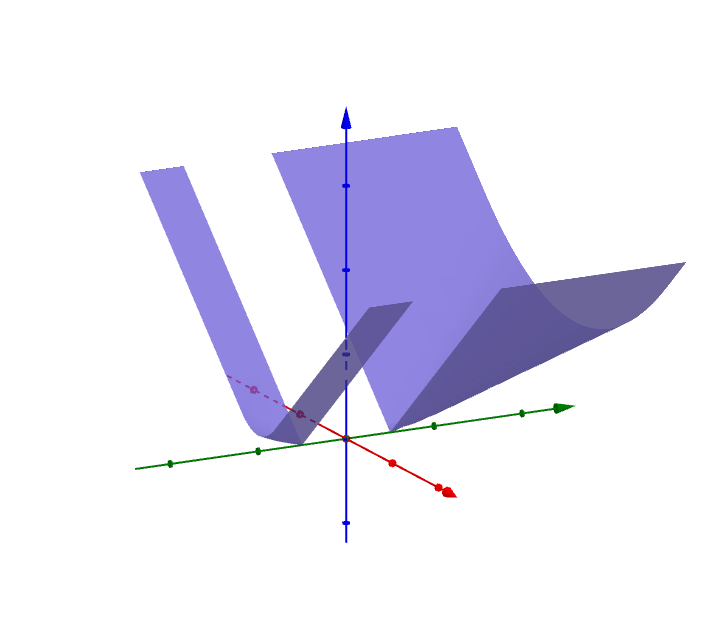}  
\vskip -9mm
 \caption{$\varphi_1\ppersp s$.}
  \label{fig:sub-first}
\end{subfigure}
\begin{subfigure}{.5\textwidth}
  \centering
  \includegraphics[width=.8\linewidth]{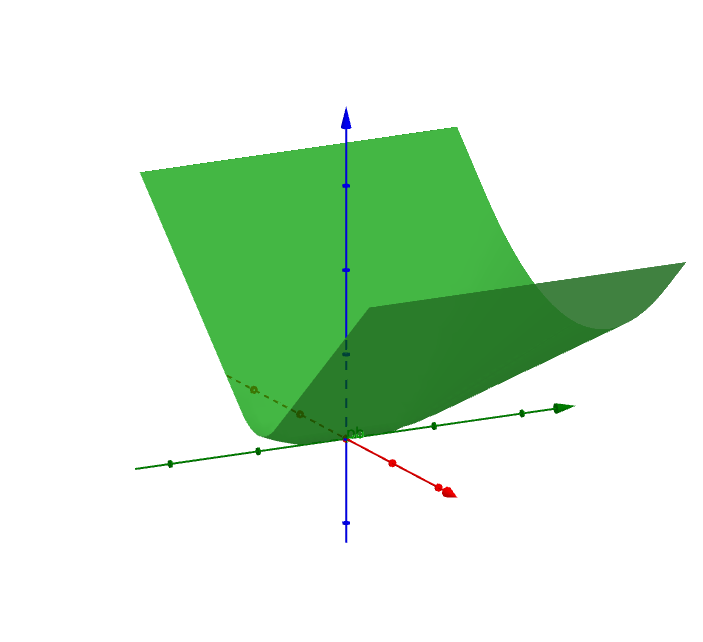}  
\vskip -9mm
  \caption{$\varphi_1\persp s$.}
  \label{fig:sub-second}
\end{subfigure}

\begin{subfigure}{.5\textwidth}
  \centering
  \includegraphics[width=.8\linewidth]{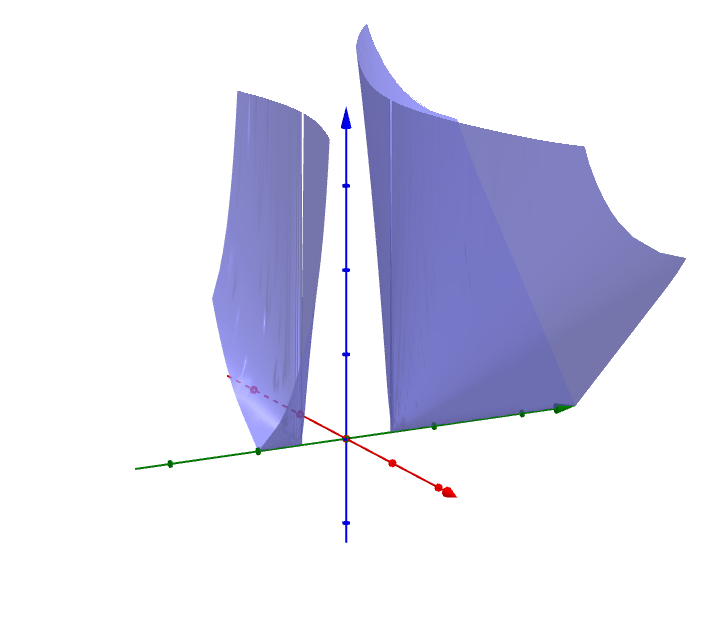}  
\vskip -9mm
  \caption{$\varphi_2\ppersp s$.}
  \label{fig:sub-third}
\end{subfigure}
\begin{subfigure}{.5\textwidth}
  \centering
  \includegraphics[width=.8\linewidth]{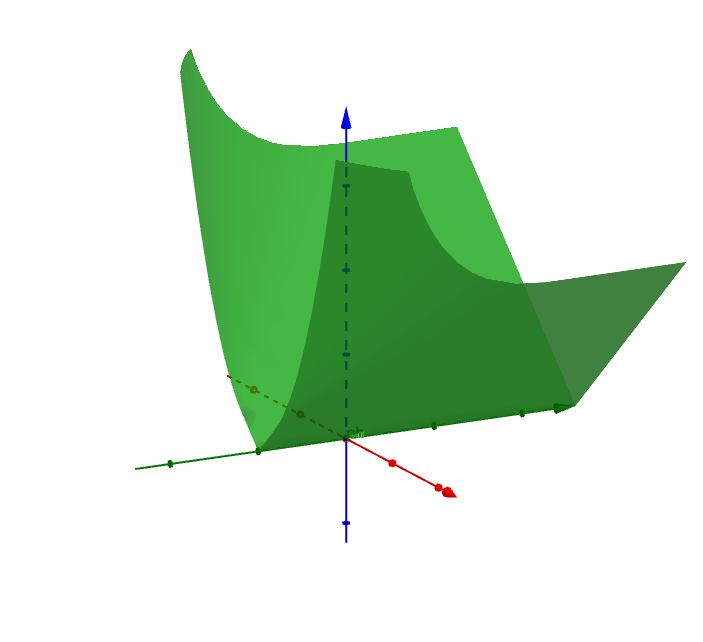}  
\vskip -9mm
  \caption{$\varphi_2\persp s$.}
  \label{fig:sub-fourth}
\end{subfigure}

\begin{subfigure}{.5\textwidth}
  \centering
  \includegraphics[width=.8\linewidth]{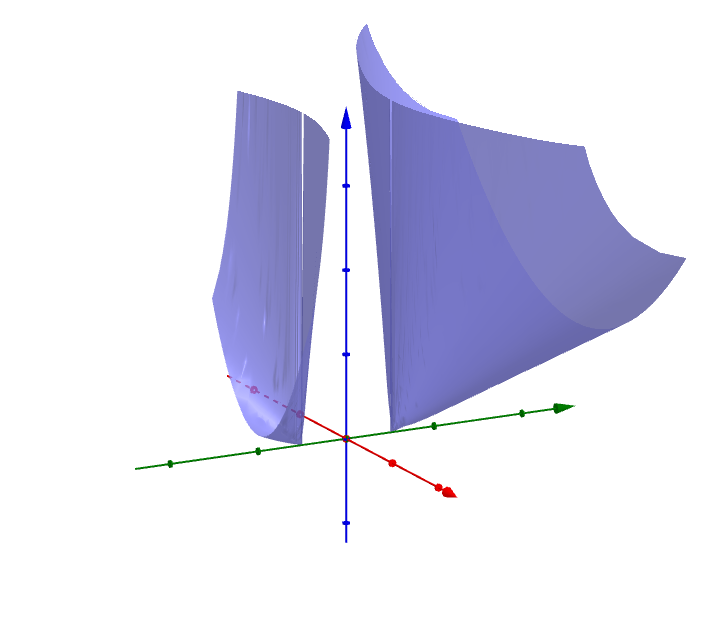}  
\vskip -9mm
  \caption{$\varphi_3\ppersp s$.}
  \label{fig:sub-fifth}
\end{subfigure}
\begin{subfigure}{.5\textwidth}
  \centering
  \includegraphics[width=.8\linewidth]{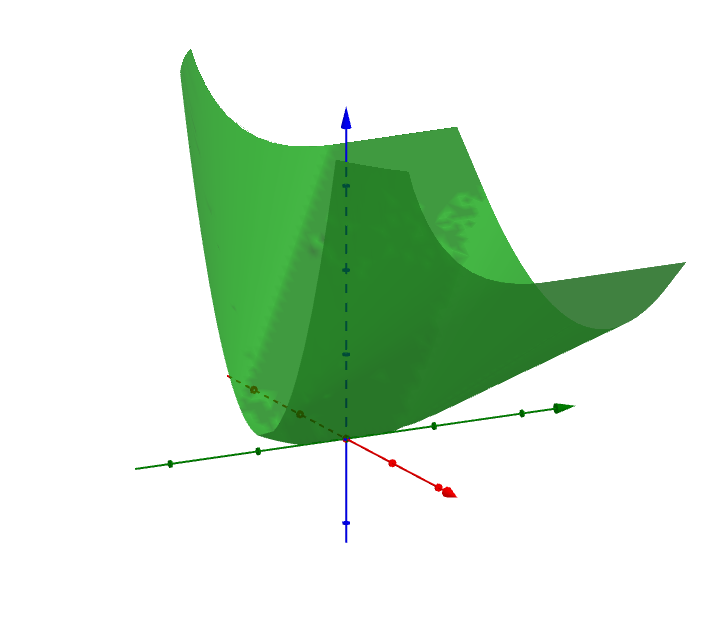}  
\vskip -9mm
  \caption{$\varphi_3\persp s$.}
  \label{fig:sub-sixth}
\end{subfigure}
\vskip 9mm

\caption{Plots of $\varphi_i\ppersp s$ (left) and 
$\varphi_i\persp s$ (right) for $i\in\{1,2,3\}$ in
Examples~\ref{ex:z1}-\ref{ex:z3} with $p=2$, $\alpha=1$, and
$\beta=1/2$. The $x$-axis is in red and the $y$-axis in green.}
\label{fig:exz}
\end{figure}

\begin{example} 
\label{ex:qneq1}
Let $\XX$ and $\varphi_3$ be as in Example~\ref{ex:z3},
let $\varphi_1$ be as in Example~\ref{ex:z1}, and
let $\varphi_2$ be as in Example~\ref{ex:z2}.
Recall that $(\varphi_3^*)^{-1}(\RMM)\neq\emp$, 
$(\varphi_3^*)^{-1}(\RPP)\neq\emp$,
$\ebase{\varphi_3}=\varphi_1$, and
$\ehaute{\varphi_3}=\varphi_2$.
Suppose that $\YY=\RR$, let $1\neq q\in\RPP$, and set 
\begin{equation}
\label{ex:puissance_q}
s\colon\RR\to\RX\colon y\mapsto\begin{cases}
y^q,&\text{if} \;\: y\geq 0; \\
\pinf,&\text{if} \;\: y<0.
\end{cases}
\end{equation}
Since $\cam\pbas{(-s)}=\emp$ for $q>1$, it follows from
\eqref{e:bas}, \eqref{e:haut}, Lemma~\ref{l:12}\ref{l:12ii}, and 
Lemma~\ref{l:24}\ref{l:24ii} that
\begin{equation}
\haut{s}\colon y\mapsto 
\begin{cases}
0,&\text{if}\:\: y\geq 0\:\:\text{and}\:\:q<1;\\
y^q,&\text{if}\:\: y\geq 0\:\:\text{and}\:\: q>1;\\
\pinf,&\text{if}\:\: y<0
\end{cases}\quad \text{and}\quad 
-\bas{(-s)}\colon y\mapsto 
\begin{cases}
y^q,&\text{if}\:\: y\geq 0\:\:\text{and}\:\:q<1;\\
\minf,&\text{if}\:\: y<0\:\:\text{and}\:\:q<1;\\
\pinf,&\text{if}\:\: q>1.
\end{cases}
\end{equation}
Hence, we derive $\varphi_3\persp s$ from
Theorem~\ref{t:55}\ref{t:55vb} for $q<1$, and from
Theorem~\ref{t:55}\ref{t:55va} for $q>1$ (see
Figure~\ref{fig:ex641}).
\end{example}

\begin{figure}
\begin{subfigure}{.5\textwidth}
  \centering
  \includegraphics[width=.7\linewidth]{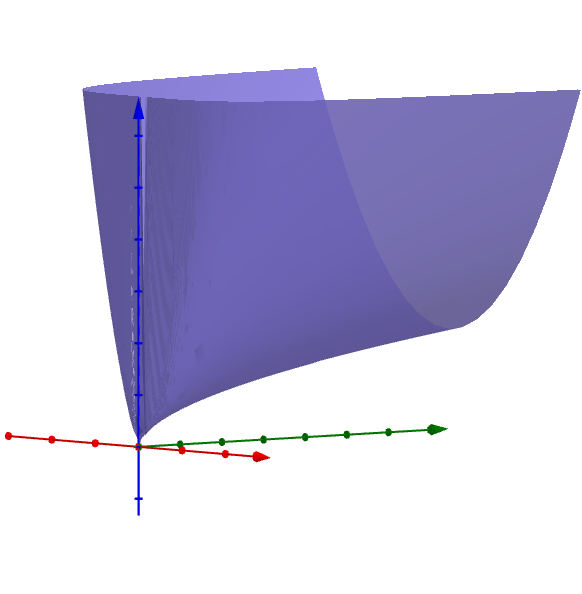}  
\vskip -9mm
 \caption{$\varphi_3\ppersp s$ with $p=1/q=2$.}
  \label{fig:ex640pp}
\end{subfigure}
\begin{subfigure}{.5\textwidth}
  \centering
  \includegraphics[width=.7\linewidth]{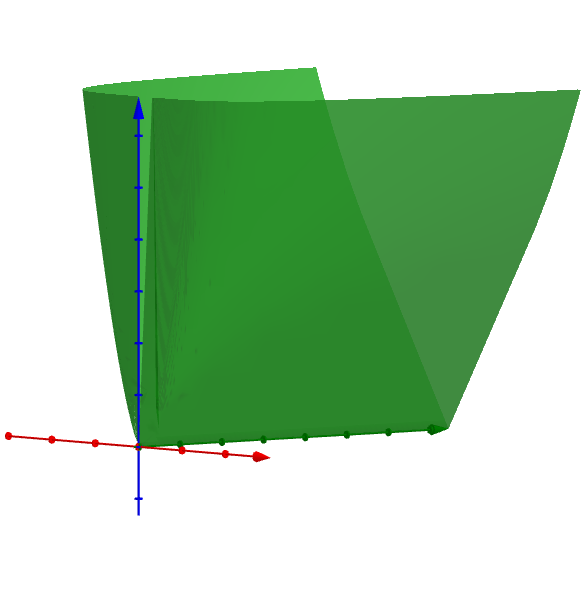}  
\vskip -9mm
  \caption{$\varphi_3\persp s$ with $p=1/q=2$.}
  \label{fig:ex640p}
\end{subfigure}
\vskip 19mm
\begin{subfigure}{.5\textwidth}
  \centering
  \includegraphics[width=.7\linewidth]{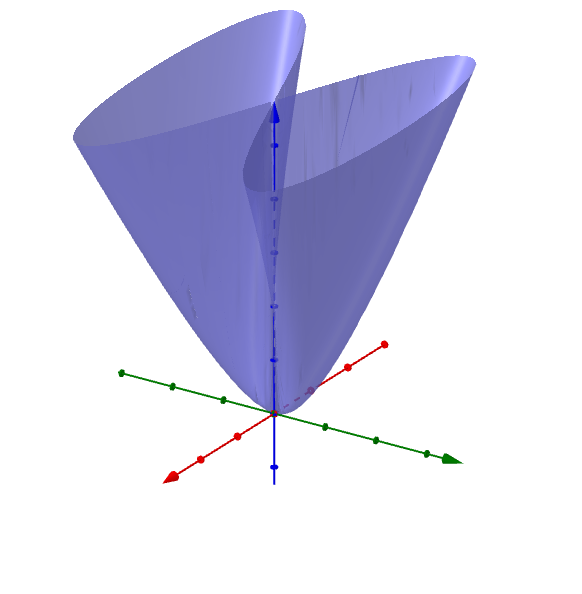}  
\vskip -9mm
  \caption{$\varphi_3\ppersp s$ with $p=q=2$.}
  \label{fig:ex641pp}
\end{subfigure}
\begin{subfigure}{.5\textwidth}
  \centering
  \includegraphics[width=.7\linewidth]{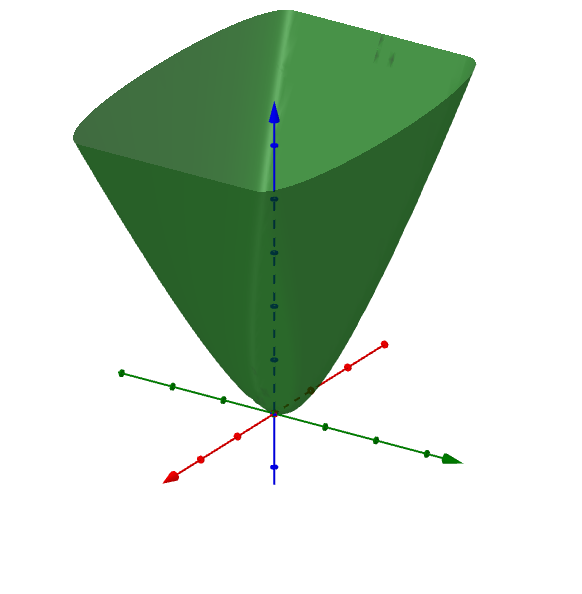}  
\vskip -9mm
  \caption{$\varphi_3\persp s$ with $p=q=2$.}
  \label{fig:ex641p}
\end{subfigure}
\vskip 9mm

\caption{Plots of $\varphi_3\ppersp s$ (left) and 
$\varphi_3\persp s$ (right) in Example~\ref{ex:qneq1}. 
The $x$-axis is in red and the $y$-axis in green.}
\label{fig:ex641}
\end{figure}

We now turn our attention to specific applications by considering
integral functions of the form
\begin{equation}
(\mathsf{x},\mathsf{y})\mapsto\displaystyle{\int_\Omega}
\big(\varphi_{\omega}\persp s_{\omega}\big)
\big(\mathsf{x}(\omega),\mathsf{y}(\omega)\big)\mu(d\omega),
\end{equation}
where the integrand is a perspective function with nonlinear
scaling in the sense of Definition~\ref{d:persp}.

\begin{example}
\label{ex:MFG}
Let $p\in\left]1,\pinf\right[$ and $q\in\rzeroun$. Suppose 
that $\XX=\RR^N$ is normed with $\|\cdot\|$, $\YY=\RR$, 
$\varphi\colon\XX\to\RX\colon x\mapsto\|x\|^p/p$, and
\begin{equation}
\label{e:v}
s\colon\YY\to\RXX\colon y\mapsto 
\begin{cases}
y^q,&\text{if}\:\:y\geq 0;\\
\minf,&\text{if}\:\:y<0.
\end{cases}
\end{equation}
Let $T\in\RPP$, set $\UU=(L^1 ([0,T]\times\RR^d))^N$, set 
$\VV=L^1([0,T]\times\RR^d)$, and consider the integral function 
\begin{equation}
\label{e:defJ}
\Phi\colon\UU\oplus\VV\to\RX\colon
(m,\varrho)\mapsto\displaystyle{\int_{0}^{T}\int_{\RR^d} 
(\varphi\persp s)
\big(m(t,\xi),\varrho(t,\xi)\big)dtd\xi}.
\end{equation}
In optimal mass transportation theory, $m$ and $\varrho$ represent
the momentum and the density of particles, respectively, and
$m/\varrho$ represents their velocity \cite{Bren00,Vill03}. In the
case when $p=2$ and $q=1$, $\varphi\persp s$ is a classical
perspective (see \eqref{e:0}) and the function \eqref{e:defJ}
is related to the dynamical formulation of the $2$-Wasserstein
distance \cite{Bren00,Vill03}. Based on this formulation, convex
optimization methods are proposed in
\cite{Bena16,Carr22} to approximate the iterates
of the so-called JKO scheme \cite{Jord98} for gradient
flows in the space of probability measures. When $q\neq 1$,
\eqref{e:defJ} appears in optimal transportation based on
$p$-Wasserstein distances with nonlinear mobilities 
\cite{Card13,Carr10,Dolb09}
and in the optimal control of McKean--Vlasov systems with
congestion \cite{Ach16a}. Space-dependent potentials
$(\varphi_{\xi})_{\xi\in\Xi}$, where $\Xi\subset\RR^d$, are also 
found \cite{Bena15,Bric18,Carda15}, where they lead to
functions of the form 
\begin{equation}
\label{e:defJ2}
\Psi\colon\UU\oplus\VV\to\RX\colon
(m,\varrho)\mapsto\displaystyle{\int_{0}^{T}\int_{\Xi} 
(\varphi_\xi\persp s)
\big(m(t,\xi),\varrho(t,\xi)\big)dtd\xi}.
\end{equation}
Theorem~\ref{t:55} provides conditions under which
$(\varphi_{\xi}\persp s)_{\xi\in\Xi}$ is a family of functions in
$\Gamma_0(\XX\oplus\YY)$. Note that in
\cite{Bena15,Bric18,Carda15}, $q=1$ and we are therefore dealing
with classical perspectives (see Example~\ref{ex:307}). Our
nonlinear setting allows us to employ \eqref{e:defJ2} with 
$q<1$ and more structured 
space-dependent potentials. For instance, in the context of 
optimal transport theory, consider
\begin{equation}
(\forall\xi\in\Xi)\quad
\varphi_{\xi}\colon\RR^N\to\RX\colon x\mapsto
\|x\|^p/p+\iota_{C(\xi)}(\|x\|)+h(\xi),
\end{equation}
where $C(\xi)\subset\RP$ is an interval representing a constraint
on the speed of particles located at $\xi$ and $h$ is a spatial
penalization term.
For every $\xi\in\Xi$ such that $\inf C(\xi)>0$, we have 
$(\varphi_{\xi}^*)^{-1}(\RPP)\neq\emp$ and
$(\varphi_{\xi}^*)^{-1}(\RMM)\neq\emp$, and Theorem~\ref{t:55}
is needed to compute $\varphi_\xi\persp s$. An illustration is
provided in Figure~\ref{fig:MFG}.
Another type of scaling function in \eqref{e:defJ} is proposed in 
\cite{Bren02}, namely the concave function
\begin{equation}
s\colon\YY\to\RXX\colon {y}\mapsto 
\begin{cases}
\dfrac{y(1-y)}{\alpha(1-y)+\beta y},
&\text{if}\:\:{y}\in[0,1];\\[5mm]
\minf,&\text{otherwise},
\end{cases}
\end{equation}
where $(\alpha,\beta)\in\RPP^2$.
\end{example}

\begin{figure}
\begin{center}
\includegraphics[scale=0.3]{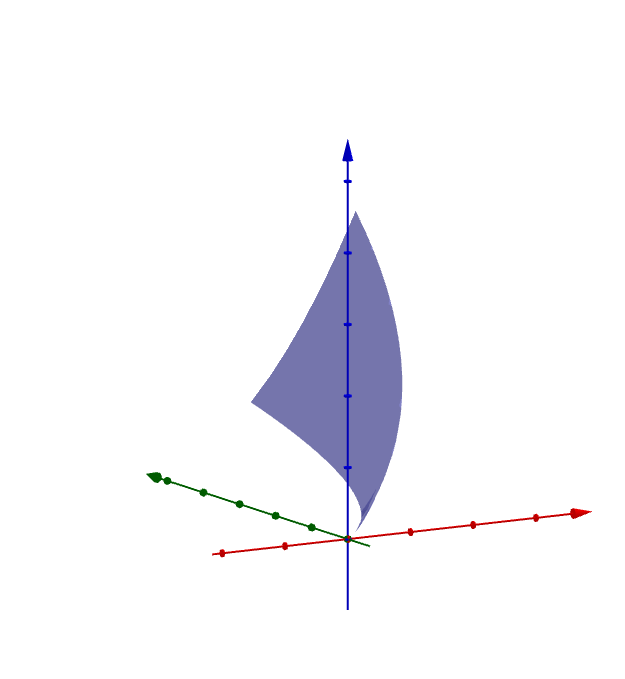}
\includegraphics[scale=0.3]{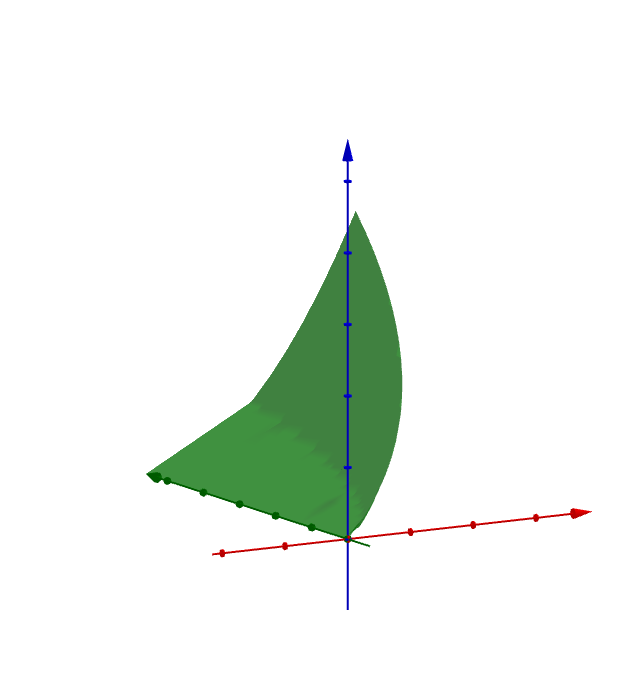}
\end{center}
\caption{Plot of $\varphi\ppersp s$ (left) and $\varphi\persp s$
(right) in Example~\ref{ex:MFG} for 
$\XX=\YY=\RR$, $\varphi=|\cdot|^2/2+\iota_{[1,2]}$, 
and $s\colon y\mapsto\sqrt{y}$ if $y\geq 0$.
The $x$-axis is in red and the $y$-axis in green.}
\label{fig:MFG}
\end{figure}

\begin{example} 
Let $U$ be a finite set, suppose that $\XX=\RR$ and $\YY=\RR^2$,
let $s\colon\YY\to[\minf,\pinf]$, and, for every $(u_1,u_2)\in
U^2$, let $\varphi_{u_1,u_2}\in\Gamma_0(\XX)$. Furthermore, set
$\UU=L^{2}([0,1];\RR^{U\times U})$, $\VV=L^1([0,1];\RR^{U})$, and 
\begin{align}
\label{e:maasgen}
\Phi\colon\UU\oplus\VV&\to\RX\nonumber\\
(m,\varrho)&\mapsto\displaystyle{\int_{0}^{1}} 
\sum_{u_1\in U}\sum_{u_2\in U}\big(\varphi_{u_1,u_2}\persp s\big)
\big(m(t,u_1,u_2),\varrho(t,u_1),
\varrho(t,u_2)\big)dt.
\end{align}
Theorem~\ref{t:55}
provides conditions under which, for every $(u_1,u_2)\in U^2$,
$\varphi_{u_1,u_2}\persp s\in \Gamma_0(\XX\oplus\YY)$.
In the particular case when, for every 
$(u_1,u_2)\in U^2$, 
$\varphi_{u_1,u_2}=K(u_1,u_2)\pi(u_1)\varphi$, 
where $\varphi\colon\XX\to\RX\colon x\mapsto |x|^2/2$,
$K\colon U\times U\to\RR$ is an
irreducible and reversible Markov kernel on $U$, and
$\pi\colon U\to\RR$ is the associated stationary distribution,
\eqref{e:maasgen} reduces to
\begin{align}
\Phi\colon (m,\varrho)\mapsto\displaystyle{\int_{0}^{1}} 
\sum_{u_1\in U}\sum_{u_2\in U}
\big(\varphi \persp s\big)
\big(m(t,u_1,u_2),\varrho(t,u_1),
\varrho(t,u_2)\big)
K(u_1,u_2)\pi(u_1) dt,
\end{align}
which appears in \cite{Maas11}. Under some additional conditions on
$s$, satisfied for instance by the logarithmic mean 
\begin{equation}
s\colon(y_1,y_2)\mapsto
\begin{cases} 0, & \text{if}\; 
(y_1,y_2)\in(\{0\}\times\RP)\cup(\RPP\times\{0\});\\
y_1, & \text{if}\; y_1=y_2\in\RPP;\\
\dfrac{y_2-y_1}{\log(y_2)-\log(y_1)}, & \text{if}\; 
(y_1,y_2)\in\RPP\times\RPP\;\text{and}\;
y_1\neq y_2;\\
\minf, & \text{otherwise,}
\end{cases} 
\end{equation}
and by the geometric mean 
\begin{equation}
s\colon (y_1,y_2)\mapsto
\begin{cases} 
\sqrt{y_1y_2}, &\text{if}\; (y_1,y_2)\in\RP\times\RP; \\
\minf, & \text{otherwise},
\end{cases}
\end{equation}
the function $\Phi$ is used in \cite{Maas11} to construct a
distance on the set of probability
densities on $U$ with respect to $\pi$.
\end{example}

\begin{example}
\label{ex:fi}
One of the oldest instances involving standard perspective
functions is the Fisher information of a differentiable probability
density $\mathsf{y}\colon\RR^N\to\RPP$ \cite{Fish25}, that is,
\begin{equation}
\label{e:fisher}
\Psi(\mathsf{y})=\int_{\RR^N}\frac{\|\nabla
\mathsf{y}(\omega)\|_2^2}{\mathsf{y}(\omega)}d\omega,
\end{equation}
where $\|\cdot\|_2$ is the standard Euclidean norm on $\RR^N$.
Going back to Definition~\ref{d:persp}, given a nonempty open
set $\Omega\subset\RR^N$, \eqref{e:fisher} can be formalized as an
instance of the function 
\begin{equation}
\label{e:y7}
\Psi\colon W^{1,r}(\Omega)\to\RX\colon
\mathsf{y}\mapsto\displaystyle{\int_\Omega}
\big(\varphi\persp s\big)
\big(\nabla \mathsf{y}(\omega),\mathsf{y}(\omega)\big)d\omega,
\end{equation}
where $r\in\left[1,\pinf\right[$, $\XX=\RR^N$, $\YY=\RR$, 
$\varphi=\|\cdot\|^2_{2}$,
and $s\colon y\mapsto y$. More generally,
assume that $\Gamma_0(\XX)\ni\varphi\geq 0$ and that
$\Gamma_0(\YY)\ni -s\leq 0$ satisfies $s^{-1}(\RPP)\neq\emp$. 
Then $\cam\pbas{(-s)}\neq\emp$ and 
Theorem~\ref{t:55}\ref{t:55i-}
asserts that $\varphi\persp s\in\Gamma_0(\XX\oplus\YY)$. In turn,
the linearity and the continuity of $\mathsf{y} \mapsto(\nabla
\mathsf{y},\mathsf{y})$ imply
that $\Psi\in\Gamma_0(W^{1,r}(\Omega))$.
For instance, let $\|\cdot\|$ be a norm on $\RR^N$, let 
$p\in\left]1,\pinf\right[$, take 
$\gamma\in\left]1/p,1\right]$, set 
$q=(\gamma p-1)/(p-1)\in\left]0,1\right]$, and define 
\begin{equation}
\label{e:e_q}
\varphi=\|\cdot\|^p\quad\text{and}\quad
s\colon\YY\to\left[\minf,\pinf\right[ \colon y\mapsto\begin{cases}
y^q,&\text{if}\;\: y\geq 0; \\
\minf,&\text{if}\;\: y<0.
\end{cases}
\end{equation}
To make \eqref{e:y7} explicit in this scenario, let us introduce
\begin{equation}
\ln_{\gamma}\colon\RR\to\left[\minf,\pinf\right[ \colon {y}
\mapsto 
\begin{cases}
\dfrac{y^{1-\gamma}-1}{1-\gamma},&\text{if}\:\:\gamma\neq
1\:\:\text{and}\:\:y\in\RPP;\\
\ln y,&\text{if}\:\:\gamma=1\:\:\text{and}\:\:y\in\RPP;\\
\minf,&\text{if}\:\: y\in\RM
\end{cases}
\end{equation} 
and note that 
$(\forall y\in\RPP)$ $(\ln_{\gamma})'(y)=1/y^{\gamma}$.
Let $\mathsf{y}\in W^{1,r}(\Omega)$, set
$\Omega_0=\menge{\omega\in\Omega}{\mathsf{y}(\omega)=0}$, and
set
$\Omega_+=\menge{\omega\in\Omega}{\mathsf{y}(\omega)>0}$. 
Then, by Corollary~\ref{c:305}\ref{c:305iii} and
\cite[Proposition~5.8.2]{Atto06}, 
if $\mathsf{y}\geq 0$ a.e.,
\begin{align}
\label{e:y8}
\displaystyle{\int_\Omega}
\big(\varphi\persp s\big)
\big(\nabla \mathsf{y}(\omega),\mathsf{y}(\omega)\big)d\omega
&=
\displaystyle{\int_{\Omega_0}}
(\rec\varphi)(\nabla \mathsf{y}(\omega))d\omega
+
\displaystyle{\int_{\Omega_+}}
s(\mathsf{y}(\omega))
\varphi\Bigg(
\dfrac{\nabla \mathsf{y}(\omega)}{s(\mathsf{y}(\omega))}
\Bigg)d\omega
\nonumber\\
&=
\displaystyle{\int_{\Omega_0}}
\iota_{\{0\}}(\nabla \mathsf{y}(\omega))d\omega
+\displaystyle{\int_{\Omega_+}}
\mathsf{y}(\omega)^{q}
\bigg\|\dfrac{\nabla\mathsf{y}(\omega)}{\mathsf{y}(\omega)^q}
\bigg\|^{p}d\omega
\nonumber\\
&=
\displaystyle{\int_{\Omega_+}}
\mathsf{y}(\omega)
\Bigg\|\dfrac{\nabla\mathsf{y}(\omega)}
{\mathsf{y}(\omega)^\gamma}\Bigg\|^{p}d\omega
\nonumber\\
&=
\displaystyle{\int_{\Omega_+}}
\mathsf{y}(\omega)\|\nabla\ln_{\gamma}\mathsf{y}(\omega)\|^{p}
d\omega.
\end{align}
Altogether, it follows from Corollary~\ref{c:305}\ref{c:305iii}
that 
\begin{equation}
\label{e:fishec}
\displaystyle{\int_\Omega}
\big(\varphi\persp s\big)
\big(\nabla\mathsf{y}(\omega),\mathsf{y}(\omega)\big)d\omega=
\begin{cases}
\displaystyle{\int_{\Omega_+}}
\mathsf{y}(\omega)\|\nabla\ln_{\gamma}\mathsf{y}(\omega)\|^{p}
d\omega,&\text{if}\;\;
\mathsf{y}\geq 0\;\;\text{a.e.};\\
\pinf,&\text{otherwise.}
\end{cases}
\end{equation}
This type of integral shows up in information theory and in 
thermostatistics \cite{Berc13,Lutw05}. In view of
Corollary~\ref{c:305}\ref{c:305iii}, our construction
\eqref{e:fishec} is guaranteed to be in
$\Gamma_0(W^{1,r}(\Omega))$, which opens a path to solve
variational problems such as those in \cite{Berc13} rigorously.
In the case when $\|\cdot\|=\|\cdot\|_2$, $p=2$, and $\gamma=q=1$,
this recovers a result of \cite{Svva18} on the Fisher information
\eqref{e:fisher}.
\end{example}

\section{Concluding remarks} 
\label{sec:7} 
We have proposed several
contributions to the theory of perspective functions with nonlinear
scaling. First, we introduce the notion of a preperspective
function and define the perspective as its largest lower
semicontinuous minorant. This construction captures the standard
case of linear scaling and guarantees properness, lower
semicontinuity, and convexity regardless of the sign of the
conjugate of the base function and of the nature of the scaling
function. Our construction necessitate the introduction of new
envelopes, called the $\blacktriangledown$ and $\blacktriangleup$
envelopes, which we have thoroughly investigated. We then compute
the Legendre conjugate of the proposed nonlinear scaled
perspectives. These conjugation formulas are central in duality
methods but they also proved to be essential to the computation of
proximity operators of perspective functions in the follow-up paper
\cite{Prox23}. Our next contribution is to provide explicit
formulas for the computation of perspective functions in a broad
range of scenarios. Finally, these notions are illustrated by
examples as well as through applications touching on areas such as
mean-field games, optimal transportation, and information theory.

\end{document}